\RequirePackage{ifpdf}
\ifpdf 
	\documentclass[pdftex]{sigma}
\else
	\documentclass{sigma}
\fi
\usepackage{
   amsmath,amsfonts,amsthm,amssymb,amsbsy}

\def\Jet{J^\infty}
\def\jet{j^\infty}
\def\F{\mathcal{F}}
\def\dif{\mathrm{d}}
\def\lap{\Delta}

\def\d{\dagger}
\def\Ct{\mathcal{C}}
\newcommand\couple[1]{\langle #1\rangle}

\newcommand\bs[1]{\boldsymbol{#1}}

\newcommand\isempty[3]{{\ifx&#1&#2\else#3\fi}}
\newcommand\term[1]{\emph{#1}}
\newcommand\dvol[1][M]{\dif^nx}
\newcommand{\hJet}[1]{\ol{\Jet_{#1}}}
\newcommand{\evder}[2][]{\partial_{#2}^{\ifx&#1&\else(#1)\fi}}
\newcommand{\ol}[1]{\overline{#1}}
\newcommand{\wh}[1]{\widehat{#1}}
\newcommand{\hdif}{\overline{\dif}}
\newcommand{\difv}[2][]{\frac{\delta #1}{\delta #2}} 
\newcommand{\ldifv}[2][]{\frac{\overleftarrow{\delta #1}}{\delta #2}} 
\newcommand{\rdifv}[2][]{\frac{\overrightarrow{\delta #1}}{\delta #2}}

\newcommand{\dift}[2][]{\frac{\dif #1}{\dif #2}}
\newcommand{\diftat}[3][]{\left.\frac{\dif #1}{\dif #2}\right|_{#2=#3}}
\newcommand{\difp}[2][]{\frac{\partial #1}{\partial #2}}
\newcommand{\dtotal}[2][]{D_{#2}\isempty{#1}{}{(#1)}}
\newcommand\schouten[1]{\lshad #1 \rshad}
\newcommand\bschouten[1]{\big\lshad #1 \big\rshad}

\newcommand{\BBR}{\mathbb{R}}\newcommand{\BBC}{\mathbb{C}}

\newcommand{\BBZ}{\mathbb{Z}}
\newcommand{\cC}{\mathcal{C}}
\newcommand{\cE}{\mathcal{E}}
\newcommand{\cH}{\mathcal{H}}
  \newcommand{\gM}{\mathfrak{M}}
  \newcommand{\gN}{\mathfrak{N}}

\newcommand{\bq}{{\boldsymbol{q}}}
  \newcommand{\bx}{x}
  \newcommand{\bu}{{\boldsymbol{u}}}
  \newcommand{\bby}{y}
  \newcommand{\bz}{z}
\newcommand{\bpi}{{\boldsymbol{\pi}}}
\newcommand{\dd}{\partial}
\newcommand{\Id}{{\mathrm d}}

  \newcommand{\lde}{\overleftarrow{\delta}\!\!}
  
  \newcommand{\ov}{\overline}
\newcommand{\BV}{{\text{\textup{BV}}}}
\newcommand{\lshad}{[\![}
\newcommand{\rshad}{]\!]}
  \newcommand{\veps}{\varepsilon}
  \newcommand{\de}{\delta}
  \DeclareMathOperator{\rank}{rank}

\DeclareMathOperator{\id}{id}
\DeclareMathOperator{\gh}{gh}

\theoremstyle{definition}
\newtheorem{counterexample}[example]{Counterexample}

\begin{document}


\ShortArticleName{On the geometry of the BV-\/Laplacian}
\ArticleName{On the geometry of the Batalin-Vilkovisky Laplacian}

\Author{Arthemy V. KISELEV and Sietse RINGERS}
\AuthorNameForHeading{A. V. Kiselev and S. Ringers}
\Address{Johann Ber\-nou\-lli Institute for Mathematics and Computer Science, University of Groningen, P.O.~Box 407, 9700~AK Groningen, The Netherlands}
\Email{\href{A.V.Kiselev@rug.nl}{A.V.Kiselev@rug.nl}, \href{mailto:S.Ringers@rug.nl}{S.Ringers@rug.nl}}

\Abstract{
We define a jet\/-\/space analog of the BV-\/Laplacian, avoiding 
delta\/-\/functions and infinite constants;
instead we show that the main properties of the BV-Laplacian and its relation to the Schouten bracket originate from the underlying jet-space geometry.
}

\Classification{%
   58A20; 
   81S10; 
   81T13
}

\thispagestyle{empty}

\subsection*{Introduction}
The Batalin\/--\/Vilkovisky (BV) Laplacian $\lap_\BV$ is a necessary ingredient in the quantization of gauge\/-\/invariant systems of Euler\/--\/Lagrange equations~\cite{BV1,BV2,HT}. Introduced after the generalization of the standard calculus on a smooth manifold (which can be viewed as the set of locations for a material point~\cite{MMKM}) to the variational calculus of fields over space-time~\cite{ehresmann,ArthemyBook,Olver}, the BV-\/Laplacian patterns upon the familiar construction $\lap = \dif_{\text{dR}}*\dif_{\text{dR}}$ (here $\dif_{\text{dR}}$ is the de~Rham differential on a manifold and $*$~is the Hodge star). In the new geometry of gauge fields, the construction of the BV-\/Laplacian $\lap_\BV$ relies on the presence of canonically conjugate pairs of variables such as fields and antifields or (possibly, higher) ghost\/-\/antighost pairs, which stem from the derivation of the Euler\/--\/Lagrange equations of motion $\delta S = 0$ and (possibly, higher generations of) Noether relations, respectively (see~\cite{BV1,BV2,BRS1,BRS2,BRST3} and~\cite{ArthemyPrague,ArthemyBook}). This relates the BV-\/Laplacian $\lap_\BV$ to the variational Schouten bracket $\lshad\,,\rshad$ (the odd Poisson bracket, or \emph{antibracket}~\cite{witten}), which measures for the operator $\lap_\BV$ its deviation from being a derivation.

The Laplace equation $\lap_\BV(F)=0$ constrains the
(products of) integral functionals~$F$ of fields and antifields in such
a way that Feynman's path integral of a Laplace equation's solution~$F$
over the space of admissible fields is essentially independent of the
non\/-\/physical antifields. As soon as the setup becomes quantum and
all objects depend also on the Planck constant~$\hbar$, the Laplace
equation
$\Delta\bigl(\mathcal{O}^\hbar\cdot\exp(iS^\hbar_{\text{BV}}/\hbar)\bigr)=0$
selects the observables~$\mathcal{O}^\hbar$ (here $i^2=-1$ and
$S^\hbar_{\text{BV}}$ is the extension in powers of~$\hbar$ for the full
BV-\/action~$S_{\text{BV}}=S+\ldots$ of a given gauge\/-\/invariant
model $\delta S=0$). This approach yields the quantum master equation
upon~$S^\hbar_{\text{BV}}$ and creates the cohomology groups with
respect to the quantum BV-\/differential
$\Omega = -i\hbar\Delta+[\![S^\hbar_{\text{BV}},\cdot\,]\!]$.
The observables are $\Omega$-\/closed,
$\Omega(\mathcal{O}^\hbar)=0$; Feynman's path integral is
then used to calculate their expectation values and correlations.

On top of the difficulties which are immanent to a definition of the path integral, 
the BV-\/Lap\-la\-cian $\lap_\BV$ itself often suffers from a necessity to be regularized manually if one wishes to avoid the otherwise appearing ``infinite constants'' or Dirac's delta-distributions (e.g.,\ see\:\cite{CF}, in which such constants appear).\ We now argue that such difficulties,\ leading to a necessity to regularize the object\:$\lap_\BV$,\ are brought into the picture by an incomplete 
utilization 
of the geometry at hand -- 
so that the 
manual re\-gu\-la\-ri\-za\-ti\-on procedure is in fact redundant: we claim that no sources of divergence are built into the genuine definition of $\lap_\BV$. The intrinsic self\/-\/regularization of the BV-\/Laplacian, the necessity of which was long stated 
in the mathematical physics literature, is 
achieved 
in this paper.

We approach the regularization problem for the BV-\/Laplacian~$\lap_\BV$ in terms of the geometry which properly takes into account (i) the Euler\/--\/Lagrange equations of motion $\cE=\bigl\{\delta S=\cE_\alpha\,\delta\phi^\alpha=0\bigr\}$ upon the physical fields~$\phi^\alpha$, (ii) Noether's identities $\Phi=\bigl\{\Phi^\mu\bigl(x,[\phi^\alpha],[\cE_\beta]\bigr)\equiv0\bigr\}$ that hold regardless of 
whether the fields satisfy the equations of motions,
and possibly (iii) higher generations of constraints $\chi_i=\bigl\{\chi_i^b\bigl(x,[\phi^\alpha],[\cE_\beta],[\Phi^\mu],\ldots,[\chi_{i-1}^\nu]\bigr)\equiv0\bigr\}$ between the equations~$\cE$, Noether's identities~$\Phi$, etc.\ (see~\cite[Ch.\,2]{ArthemyBook}). The BV-\/setup is then achieved in several standard steps~\cite{BV1,BV2}. We first recall that differential equations of any nature (e.g., $\cE_\alpha=0$ or $\Phi^\mu=0$) map their arguments to spaces of sections of auxiliary bundles so that the pre\/-\/image of the zero section, which stands in the equations' right\/-\/hand sides, presents particular interest; in section~\ref{SecSpace} below we recall in detail a construction of the modules $P_0\ni\cE$, $P_1\ni\Phi$, $P_2\ni\chi_2$, $\ldots$, $P_\lambda\ni\chi_\lambda$ of equations, identities, and higher generations of constraints. For every $i=0,\ldots,\lambda$ we then introduce the dual space $\widehat{P}_i$ of linear, form\/-\/valued functions on~$P_i$ by recognizing each~$\widehat{P}_i$ as the module of sections of the dual bundle induced over the product over~$M\ni x$ of the infinite jet space~$J^\infty(\pi)$ of physical fields~$\phi^\alpha$ and, whenever appropriate, jet spaces that encode Noether's identities~$\Phi$, $\chi_2$, $\ldots$,~$\chi_{i-1}$. Thirdly, by using the concept of \emph{neighbours} (c.f.~\cite{Voronov}) we reverse the (ghost-)\/parity of the duals and obtain the odd\/-\/parity modules~$\Pi\widehat{P}_i$. 

In agreement with the BV-\/procedure, we now change the scene by `forgetting' that differential constraints 
$\Phi$, $\chi_2$, $\ldots$, $\chi_\lambda$ involve the fields~$\phi^\alpha$, equations of motion, and lower\/-\/order generations of identities.\footnote{Indeed, the equations $\mathcal{E}_\alpha = 0$ or equations between equations $\Phi^\mu = 0$, or their higher analogs are (non)linear differential operators that map the fields $\phi^\alpha$, equations $\mathcal{E}_\beta$, etc., to sections of some vector bundles with coordinates $\mathcal{E}_\alpha$, $\Phi^\mu$, etc., along the fibres. It is a Whitney sum of those bundles which we now take, obtaining the bundle $\bpi$ and also introducing its dual $\wh\bpi$ (see section~\ref{SecSpace} below).} Namely, we replace the respective objects of opposite ghost\/-\/parities by introducing the full set of BV-\/variables~$(\bq,\bq^\d)$ and constructing the $\BBZ_2$-\/graded bundle~$\pi_\BV$ over the space\/-\/time~$M$; we retain the physical fields~$\phi^\alpha$ among~$q^a$'s but also produce\footnote{The traditional notation is $\bq=\{\phi^\alpha,\gamma^{\d,\mu},c^\d_\nu,\ldots\}$ for even\/-\/parity BV-\/variables stemming from~$P_i$'s whereas $\bq^\d=\{\phi^\d_\alpha,\gamma_\mu=\bigl(\gamma^\d\bigr)^\d_\mu,c^\nu=\bigl(c^\d\bigr)^{\d,\nu},\ldots\}$ are odd\/-\/parity BV-\/variables coming from the odd neighbours~$\Pi\widehat{P}_i$.}
the odd \emph{antifields}~$\phi^\d_\alpha$ as the respective part of~$\bq^\d$, and also the odd \emph{ghosts}~$\gamma_\mu$ and even\/-\/parity \emph{antighosts}~$\gamma^{\d,\mu}$ or higher ghost\/-\/antighost pairs $c^\alpha\leftrightarrow c^\d_\alpha$ (see also~\cite{l*-coverings} in this context). This argument yields the BV-\/zoo of canonically conjugate pairs $q^a\leftrightarrow q^\d_a$ of (anti)\/fields or (anti)\/ghosts and their derivatives along~$M$, that is, along the base manifold in the bundle~$\pi$ of physical fields. 

In the course of the scene change, the form\/-\/valued coupling between the modules~$P_i$ of differential equations of all sorts and the neighbours~$\Pi\widehat{P}_i$ takes in the BV\/-\/setup the shape of an $\BBR$-{} or $\BBC$-\/valued coupling~$\langle\,,\,\rangle$ between the variations of canonically conjugate variables: $\langle\delta q^a,\delta q^\d_b\rangle=+1\cdot\delta^a_b=
-\langle\delta q^\d_b,\delta q^a\rangle$ (c.f.~\cite[\S37]{MMKM}). Given this coupling and a proper count of the ghost $\BBZ_2$-\/parity (as given by the operator~$\Pi$) and $\BBZ$-\/grading (stemming from differential forms such as the variations), the introduction of an odd Poisson bracket~\eqref{EqDefSchouten} is immediate. We proceed with a definition of the BV-\/Laplacian~$\lap_\BV$ by Eq.~\eqref{eq:DefLaplacian}; in a purely even initial setup, this yields a familiar coordinate expression for this operator:
\[
\lap_\BV=\left.\left(\frac{\overleftarrow{\delta}}{\delta q^a}\circ
 \frac{\overleftarrow{\delta}}{\delta q^\d_a}\right)\right|_{(x,[\bq],[\bq^\d])}.
\]
The core idea of its 
self\/-\/regularization roots in the congruence of attachment points for $\langle\delta\bq|$ and $|\delta\bq^\d\rangle$ in the course of evaluation of the BV-\/Laplacian of a functional at a section of the bundle~$\pi_\BV$ (that is, for a given configuration of even physical fields~$\phi^\alpha$, antighosts~$\gamma^\d_a$ etc., and for a certain choice of the antifields $\bq^\d$; the path integrals of $\mathcal{O}^\hbar\cdot\exp(iS^\hbar_\BV/\hbar)$ in the Euler\/--\/Lagrange model at hand is in fact independent of~$\bq^\d$ whenever the integrands $\mathcal{O}^\hbar\cdot\exp(iS^\hbar_\BV/\hbar)$ containing the quantum BV-action $S^\hbar_\BV$ satisfy the Laplace equation).

A regularization of the BV-\/Laplacian is often described in the literature in terms of Dirac's $\delta$-\/distributions; however, we claim that the intrinsic self\/-\/regularization scheme for~$\Delta_\BV$ 
is parallel to the identically\/-\/zero value of the coupling for usual covectors and vectors attached at two different points of a manifold.

This note is structured as follows. After introducing some notation and conventions, we inductively construct the space $\ol{\mathfrak{M}}^n(\pi_\BV)$ of local functionals 
(some of which, after a transition to the quantum setup, are the observables $\mathcal{O}$
) 
by using integral functionals such as the BV-\/action~$S_\BV$ as building blocks. We then recall the definition of the variational Schouten bracket $\lshad\,,\rshad$ as the odd Poisson bracket on the space $\ol{H}^n(\pi_\BV)$ of building blocks, then extending the antibracket to the entire linear space $\ol{\mathfrak{M}}^n(\pi_\BV)$ of local functionals by the Leibniz rule. We formulate a geometric definition of the BV-\/Laplacian in such a way that $\lap_\BV$~requires no external regularization; the operator $\lap_\BV$ satisfies -- not just formally but in earnest -- the standard product rule
formula~\eqref{eq:laplacian-on-product} that involves the variational Schouten bracket $\lshad\,,\rshad$ on $\ol{\mathfrak{M}}^n(\pi_\BV)$. 
We then pass to a more functional analytic approach, gaining a greater flexibility in dealing with variations of functionals, and we verify the main properties of the BV-Laplacian. Finally, we illustrate our reasoning in the quantum setup by a standard derivation of the quantum master equation and by a construction of the quantum BV-differential.

This note 
may be regarded as a continuation of the review~\cite{ArthemyPrague}; at the same time, we further the line of reasoning from~\cite{Lorentz12}.

\section{The space of BV-\/functionals}
\label{SecSpace}\label{sec:preliminaries}
Let us first fix some notation, in most cases matching that from~\cite{ArthemyPrague}, \cite{ArthemyBook} and~\cite{protaras} (for a more detailed exposition of these matters, see for example~\cite{ArthemyBook,topical,redbook,Olver}). Let $\pi \colon E \to M$ be a vector bundle of rank~$m$ over a smooth real oriented manifold of dimension~$n$; in this paper we assume all maps to be smooth. We let $x^i$~be the coordinates, with indices $i,j,k,\dots$, along the base manifold; $\phi^\alpha$ are the fibre coordinates with indices $\alpha,\beta,\gamma\dots$.

We take the infinite jet space $\pi_\infty \colon \Jet(\pi) \to M$ associated with this bundle~\cite{ehresmann,Olver}; a point from the jet space is then $\theta = (x^i,\phi^\alpha,\phi^\alpha_{x^i},\phi^\alpha_{x^i x^j},\dots,\phi^\alpha_\sigma,\dots) \in \Jet(\pi)$, where $\sigma$~is a multi\/-\/index and we put~$\phi^\alpha_\varnothing\equiv\phi^\alpha$. If~$s \in \Gamma(\pi)$ is a section of~$\pi$, or a \emph{field}, we denote with $\jet(s)$ its infinite jet, which is a section $\jet(s) \in \Gamma(\pi_\infty)$. Its value at $x \in M$ is $\jet_x(s) = (x^i, s^\alpha(x), \difp[s^\alpha]{x^i}(x), \dots,\frac{\partial^{|\sigma|}s^\alpha}{\partial x^\sigma}(x),\dots) \in \Jet(\pi)$.
We denote by $\F(\pi)$ the ring of smooth functions on the infinite jet space; the space of top\/-\/degree horizontal forms on~$J^\infty(\pi)$
is denoted by $\ol\Lambda^n(\pi)$. The highest horizontal cohomology, i.e., the space of equivalence classes of $n$-\/forms from $\ol\Lambda^n(\pi)$ modulo the image of the horizontal exterior differential $\hdif$ on $\Jet(\pi)$, is denoted by $\ol{H}^n(\pi)$; the equivalence class of $\omega \in \ol\Lambda^n(\pi)$ is denoted by $\int\omega \in \ol{H}^n(\pi)$. We will assume that the sections are such that integration by parts is allowed and does not result in any boundary terms (for example, the base manifold is compact, or the sections all have compact support, or decay sufficiently fast towards infinity, or are periodic). Lastly, the Euler operator is $\delta = \int \dif_{\mathcal{C}}\,\cdot\,$, where $\dif_\mathcal{C}$ is the Cartan differential, so that the variational derivative with respect to $\phi^\alpha$ is $\difv{\phi^\alpha} = \sum_{|\sigma|\geqslant0}(-)^\sigma \dtotal\sigma\difp{\phi^\alpha}$.

Let $\xi$~be a vector bundle over $\Jet(\pi)$, and suppose $s_1$ and $s_2$ are two sections of this bundle. They are said to be \term{horizontally equivalent} \cite{l*-coverings} at a point $\theta \in \Jet(\pi)$ if $\dtotal\sigma(s_1^\alpha) = \dtotal\sigma(s_2^\alpha)$ at $\theta$ for all multi\/-\/indices $\sigma$ and fibre\/-\/indices~$\alpha$. Denote the equivalence class by~$[s]_\theta$. The set
\[
\hJet\pi(\xi) := \{ [s]_\theta \mid s \in \Gamma(\xi), \, \theta \in \Jet(\pi) \}
\]
is called the \term{horizontal jet bundle} of~$\xi$. It is clearly a bundle over $\Jet(\pi)$, whose elements above $\theta$ are determined by all the derivatives $s^\alpha_\sigma \mathrel{{:}{=}} D_\sigma(s^\alpha)$ for all multi\/-\/indices $\sigma$ and fibre\/-\/indices~$\alpha$.

Now suppose $\zeta$~is a bundle over~$M$. Then the pullback bundle $\pi^*_\infty(\zeta)$ is a bundle over $\Jet(\pi)$, so we may consider its horizontal jet bundle $\hJet\pi(\pi_\infty^*(\zeta))$. (When no confusion is possible, we shall denote this by $\hJet\pi(\zeta)$.) The induced bundle $\pi^*_\infty(\zeta_\infty) = \pi_\infty \mathbin{{\times}_M} \zeta_\infty$ is equivalent as a bundle to $\hJet\pi(\zeta)\to M$ (see e.g.,~\cite{protaras} for a proof). This identification endows the horizontal jet space $\hJet\pi(\zeta)$ with the Cartan connection~-- namely the pullback connection on $\pi_\infty^*(\Jet(\zeta))$. Therefore there exist total derivatives $\dtotal{i}$ on the horizontal jet space $\hJet\pi(\zeta)$; in coordinates these are just the operators (denoting the fibre coordinate of $\zeta$ with $u$)\footnote{Here and in the remainder of this note we use Einstein's summation convention: whenever an 
index appears twice in an expression, once as an upper 
and once as a lower index, a summation over that index is implicitly present.}
\begin{align*}
\dtotal{i} = \difp{x^i} + \phi^\alpha_{\sigma+1_i}\difp{\phi^\alpha_\sigma} 
  + u^\beta_{\tau+1_i}\difp{u^\beta_\tau},\qquad 1\leqslant i\leqslant n.
\end{align*}
Thus, instead of the horizontal derivatives $\dtotal[u^\alpha]{\sigma}$ of sections there are now the fibre coordinates $u^\alpha_\sigma$, which have no derivatives along the fibre coordinates $\pi_\infty$, i.e., $\dd u^\beta_\tau/\dd \phi^\alpha_\sigma=0$.

Let $\zeta$ be a vector bundle over $M$. We denote with $\zeta^\d$ its dual, i.e. the vector bundle whose fibre over $x$ is the dual of the fibre of $\zeta$ over $x$. We also define $\wh\zeta = \zeta^\d\otimes\Lambda^n(M)$.

If $P$~is the space of sections of some vector bundle over $\Jet(\pi)$ (for example, $P$~is the space of sections that contains a certain differential equation $F \in P$, or $P = \Gamma(\pi_\infty^*(\pi)) \mathrel{{=}{:}} \varkappa(\pi)$, to which we return on the following page), then we denote with $P^\d = \text{Hom}_{\F(\pi)}(P,\F(\pi))$ its dual with respect to $\F(\pi)$, and with $\wh{P} = \text{Hom}_{\F(\pi)}(P,\ol\Lambda^n(\pi))$ its dual with respect to $\ol\Lambda^n(\pi)$. We will denote the couplings between elements of~$P$ and either $P^\d$ or $\wh{P}$ by $\langle p',p\rangle$, which lands in either $\F(\pi)$ or $\ol\Lambda^n(\pi)$, and where $p'$ is an element of either $P^\d$ or $\wh{P}$.

Returning to the vector bundle $\zeta$~over~$M$, we see that $\pi_\infty^*(\zeta)$ is a bundle over~$\Jet(\pi)$. Set $P_\zeta = \Gamma(\pi_\infty^*(\zeta))$. Then we have that $P_\zeta^\d = \text{Hom}_{\F(\pi)}(P,\F(\pi)) = \Gamma(\pi_\infty^*(\zeta^\d))$ and $\wh{P}_\zeta = \text{Hom}_{\F(\pi)}(P,\ol\Lambda^n(\pi)) = \Gamma\big(\pi_\infty^*(\wh\zeta)\big)$. Reversing the parity of the fibres~\cite{ArthemyPrague,ArthemyAlgebroids,Voronov} of $\wh\zeta$ using the parity\/-\/reversion operator~$\Pi$, we obtain the odd bundle $\Pi\wh{\zeta}$; if the fibre coordinates of~$\zeta$ are~$u^\alpha$, then we denote the fibre coordinates of $\Pi\wh\zeta$ by $u^\d_\alpha$ (note that the coordinate index switches position). We denote by $\Pi\wh{P}_\zeta = \Gamma(\pi_\infty^*(\Pi\wh\zeta))$
the associated space of sections on $\Jet(\pi)$.

\begin{example}\label{ExTrueAntifields} 
The space of generating sections for evolutionary derivatives is $\varkappa(\pi) = \Gamma(\pi_\infty^*(\pi)) = \Gamma(\pi)\otimes_{C^\infty(M)}\F(\pi)$. 
On the other hand, the range of the Euler operator~$\delta$ is naturally isomorphic to $\wh{\varkappa(\pi)}$ (see, e.g.,~\cite{topical,redbook}). Therefore, since in gauge theories the relevant system of differential equations $\mathcal{E}$ is 
Euler\/--\/Lagrange, i.e., $\mathcal{E} = \{\delta S=0 \in \wh{\varkappa(\pi)}\}$ for some $S \in \ol\Lambda^n(\pi)$, the space of sections~$P_0$ that contains the differential equation $\mathcal{E} = \{\delta S=0\}$ is~$P_0 \simeq \wh{\varkappa(\pi)}$.
Applying the formalism described above to $\wh{\varkappa(\pi)}$, we see that $\Pi\widehat{P}_0 = \Pi\varkappa(\pi) = \Gamma\big(\pi_\infty^*(\Pi\pi)\big)$, and taking the horizontal jet space $\hJet\pi(\wh\pi\mathbin{{\times}_M}\Pi\pi)$ we would obtain a jet bundle with coordinates $\phi^\alpha, \cE_\alpha$, $\cE^{\d,\alpha}$, and their derivatives.
\end{example}

Let us emphasize that the entire BV-\/geometry is completely determined by the initial bundle~$\pi$ of physical fields~$\phi^\alpha$ over~$M$ and by the action~$S\bigl(x,[\phi]\bigr)$ which yields the equations of motion and the structure of interrelations between Noether's identities in a gauge\/-\/invariant model at hand. For consistency, we now recall the inductive construction of the BV-\/bundle~$\pi_\BV$ over the base manifold~$M$; we consider the case when the total space~$E$ of the vector bundle~$\pi\colon E\to M$ is a non\/-\/graded, purely commutative smooth real manifold (see~\cite[Ch.\,2,\ 11]{ArthemyBook} for more detail and also~\cite{Lorentz12} where a non\/-\/graded non\/-\/commutative setup is addressed). An extension of the formalism to a $\BBZ_2$-{} or $\BBZ$-\/graded setup is standard (specifically to the Poisson sigma\/-\/model of~\cite{CF}, the sign conventions are explained in~\cite{FulpLadaStasheffSrni}).

For a given bundle~$\pi$ and action functional~$S\in\overline{H}^n(\pi)$, the $m$~equations of motion $\cE=\bigl\{\delta S=\cE_\alpha\,\delta\phi^\alpha=0\bigr\}$ upon $\phi^1,\ldots,\phi^m$ are 
   (up to a tensor multiplication by~$\overline{\Lambda}^n(\pi)$) 
sections of the induced bundle~$\pi_\infty^*(\widehat{\pi})$ over the infinite jet bundle~$J^\infty(\pi)\xrightarrow{\:\pi_{\infty}\:}M$; note that the dual bundle is introduced by using the $\overline{\Lambda}^n(\pi)$-\/valued coupling. Indeed, each equation in the Euler\/--\/Lagrange system is a variational covector: the variations $\cE_\alpha\,\delta\phi^\alpha={:}\Id_{\cC}(S){:}$ are obtained after integration by parts in the equivalence class of~$\Id_{\cC}(S)$. Because each coefficient $\cE_\alpha\bigl(x,[\phi]\bigr)$ of~$\delta\phi^\alpha$ in the left\/-\/hand side of system~$\cE$ belongs to the horizontal module~$\Gamma\bigl(\pi_\infty^*(\widehat{\pi})\bigr)=\widehat{\varkappa(\pi)}$, its neighbour is the module~$\Gamma\bigl(\pi_\infty^*(\Pi\pi)\bigr)$ of odd\/-\/parity evolutionary vector fields. We emphasize that, as soon as the equations of motion are derived from the action, their reparametrizations~$\cE'=\cE'(x,[\cE])$ --~yielding an equivalent but differently written system, e.g., in which two equations are swapped: $\cE_1'\mathrel{{:}{=}}\cE_2=0$ and $\cE_2'\mathrel{{:}{=}}\cE_1=0$~-- are in principle entirely independent from a change of coordinates~$\phi'=\phi'(x,[\phi])$ along the fibres of~$\pi$. However, it is standard to enforce the correspondence between the fields~$\phi^\alpha$ and equations~$\cE_\alpha=0$, i.e., to label the sections~$\cE\in\Gamma\bigl(\pi_\infty^*(\widehat{\pi})\bigr)$ of the bundle induced from~$\widehat{\pi}$ by~$\pi_\infty$ by using one more time the variables~$\phi^\alpha$ along fibres in~$\pi$ instead of a rigorous use of the 
dual bundle~$\widehat{\pi}\colon\widehat{E}\to M$. We thus set $\cE\in P_0\simeq\Gamma\bigl(\pi_\infty^*(\pi)\bigr)$, which implies that $\Pi\widehat{P}_0\simeq\Gamma\bigl(\pi_\infty^*(\Pi\widehat{\pi})\bigr)$; both isomorphisms are not canonical. The fibre coordinates in~$\pi_0\mathrel{{:}{=}}\pi$ and $\widehat{\pi}_0=\widehat{\pi}$ in~$P_0$ and~$\Pi\widehat{P}_0$ are the physical fields~$\phi^\alpha$ and the antifields~$\phi^\d_\alpha$, respectively. By this argument we construct the Whitney sum~$J^\infty(\pi)\mathbin{{\times}_M}J^\infty(\Pi\widehat{\pi})\to M$ of two infinite jet bundles.\footnote{A rigorous approach to the correspondence between fields and Euler\/--\/Lagrange equations, which we tracked in Example~\ref{ExTrueAntifields} above, would give us the horizontal jet bundles with fibre coordinates~$\cE_{\alpha,\sigma}$ and~$\cE^{\d,\beta}_\tau$ over~$J^\infty(\pi)$, here $|\sigma|,|\tau|\geqslant0$.}

Proceeding with the \emph{linear} Noether identities $\Phi=\bigl\{\Phi^\mu\bigl(x,[\phi],[\cE]\bigr)=0$, $1\leqslant\mu\leqslant m_1\bigr\}\in %
P_1=\Gamma\bigl(\pi_\infty\circ(\pi_\infty^*(\pi))^*(\pi_1)\bigr)$, where the linearity is specific to the geometry of Euler\/--\/Lagrange equations, we recognize the auxiliary vector bundle~$\pi_1$ over~$M$ such that $\Phi\colon\overline{J^\infty_\pi}(\pi_0)\to\Gamma(\pi_1)$ is a linear differential operator with respect to~$\cE_\alpha$ and may be nonlinear in~$\phi^\alpha$. We conventionally denote not by~$\Phi^\mu$ but by~$\gamma^{\d,\mu}$ the fibre coordinates in~$\pi_1$. Next, we construct the horizontal module~$\Pi\widehat{P}_1$ of sections by employing the 
dual, odd\/-\/parity bundle~$\widehat{\pi}_1$ with fibre coordinates~$\gamma_\mu$. 
This 
produces $m_1$~canonically conjugate ghost\/-\/antighost pairs $\gamma_\mu\leftrightarrow\gamma^{\d,\mu}$. As we did it at the previous step, we extend the setup further by taking the Whitney sum $J^\infty(\pi)\mathbin{{\times}_M}J^\infty(\Pi\widehat{\pi})\mathbin{{\times}_M}J^\infty(\pi_1)\mathbin{{\times}_M}J^\infty(\Pi\widehat{\pi}_1)\to M$ for the two generations of dual bundles.

An inductive reasoning over the finite number of higher Noether relations between the already\/-\/obtained identities gives rise to $\lambda+1$ sets of -- in total, $m+m_1+\ldots+m_\lambda$ -- pairs of dual, opposite\/-\/parity bundles~$\pi_i\leftrightarrow\Pi\widehat{\pi}_i$ of ranks~$m_i$, here~$0\leqslant i\leqslant\lambda$ and~$m_0=m$. The Whitney sum 
\begin{align*}
\bpi &= \mathop{{\bigotimes}_{\lefteqn{M}}}\limits_{i=0}^\lambda \quad\pi_i\\
\intertext{over the base~$M$ determines the Batalin\/--\/Vilkovisky (BV) superbundle}
\pi_\BV&=\bpi_\infty^*(\Pi\widehat{\bpi}_\infty)\colon \overline{J^\infty_\bpi}(\Pi\widehat{\bpi})=J^\infty(\bpi)\mathbin{{\times}_M} J^\infty(\Pi\widehat{\bpi})\longrightarrow M.
\end{align*}
We denote by
\begin{align*}
\bq&=(\phi,\gamma^\d,c^\d,\ldots)\\
\intertext{the entire set of even\/-\/parity BV-\/coordinates and by}
\bq^\d&=(\phi^\d,\gamma,c,\ldots)
\end{align*}
their odd\/-\/parity duals; the convention~\smash{$\left(\gamma^\d\right)^\d=\gamma$} resolves a minor confusion in the 
usually accep\-ted notation.

\begin{remark}
When speaking of parities~$(-)^{(\cdot)}$, we refer to the \term{ghost numbers}~$\gh(\cdot)$ and the respective parity reversion operator~$\Pi$ (for instance, $\gh(\phi^\alpha)=0$, $\gh(\phi^\d_\alpha)=1$, etc.); the entire collection of auxiliary bundles -- from~$\pi_0$ for~$P_0$ or~$\pi_1$ for~$P_1$ to~$\pi_\lambda$ for~$P_\lambda$ -- is even with respect to the ghost parity. The ghost parities of objects that belong to~$P_i$ and~$\Pi\widehat{P}_i$ at each~$i$ are always opposite: the former are ghost\/-\/parity even and the latter are odd. Should one consider an independently graded setup of fields~$\phi^\alpha$ as fibres of a superbundle~$\pi=\bigl(\pi^{\overline{0}}|\pi^{\overline{1}}\bigr)$ or a setup with $\BBZ$-\/graded differential forms taken as fields (see~\cite{CF}), and should one then introduce the action functional~$S\in\overline{H}^n(\pi)$, a count of extra signs in the equations of motion, Noether identities etc., would be tedious but straightforward (c.f. Example~\ref{ex:CF} on page~\pageref{ex:CF}).

Note further that the choice of canonically conjugate variables of opposite ghost parities is rigid in the sense that it is completely determined by the action functional~$S$: one can not take~$\phi^\d_\alpha$'s as new, odd\/-\/parity `fields' and let the old unknowns~$\phi^\alpha$ be the conjugate `antifields' (although $\Pi\circ\Pi=\id$ and $\widehat{\widehat{P_i}}=P_i$ in finite dimension) because the trivial equations of motion $\delta S/\delta\phi^\d\equiv0$ 
do not constrain the physical fields. 
\end{remark}

\begin{remark}
A division of the BV-\/coordinates $\bq\leftrightarrow\bq^\d$ by using the \emph{vector}\footnote{Note that the \emph{vector} bundle~$\pi$ is used to introduce the physical fields~$\phi^1$,\ $\ldots$,\ $\phi^m$ yet these objects could be, for example, components of a \emph{covector} $A=\sum_{\alpha=1}^n \phi^\alpha\,\Id x^\alpha$ of a gauge connection's one\/-\/form in a principal fibre bundle over~$M^n$ with a given structure Lie group (see Example~\ref{ex:YM} on page~\pageref{ex:YM}).}
bundles~$\pi_0$,\ $\ldots$,\ $\pi_\lambda$ 
leads to the understanding of their variations $\langle\delta\bq|$ and~$|\delta\bq^\d\rangle$ as covectors and odd\/-\/parity vectors, respectively (see~\cite[\S37]{MMKM}). 
At the same time, the following heuristic argument motivates a choice --in what follows, irrelevant-- of an upper or lower location of the index that enumerates fibre components in the bundle~$\pi_i$ or~$\Pi\widehat{\pi}_i$ at fixed~$i$. 
Namely, we remember that $\phi^\alpha$~is a fibre coordinate in the initial vector bundle~$\pi$, whence the odd\/-\/parity antifields~$\phi^\d_\beta$ along the fibre of~$\Pi\wh\pi$ can be viewed as covectors. Therefore, the variation $\langle\delta\phi^\alpha|$ is indeed a covector and $|\delta\phi^\d_\beta\rangle$ automatically becomes a vector. 
Next, let us recall that the Euler\/--\/Lagrange equations of motion $\cE=\{\delta S=\mathcal{E}_\alpha\,\delta\phi^\alpha=0\}\in \wh{\varkappa(\pi)}$
are variational covectors by construction.
The first generation of linear Noether's relations $\Phi=\bigl\{\Phi^\mu\bigl(x,[\phi^\alpha],[\mathcal{E}_\beta]\bigr)=0\bigr\}\in P_1$ between the equations of motion $\mathcal{E} \in P_0$ is determined by linear differential operators in total derivatives acting on the components of the variational covector 
$\mathcal{E}=\delta\mathcal{S}$. Consequently, the antighosts $\gamma^{\d,a}$ are even vectors and their parity\/-\/reversed duals~$\gamma_a$ 
are odd covectors, so that
$\langle\delta\gamma^{\d,\alpha},\delta\gamma_\beta\rangle=+1\cdot\delta^\alpha_\beta$. It is similar for the second generation of ghost\/-\/antighost pairs 
$c^\alpha 
\leftrightarrow c^\d_\beta$, 
etc.\ (see~\cite[Ch.\,2 and~11]{ArthemyBook}).

We finally notice that the fact that an object may be a variational (co)vector relative to the initial vector bundle~$\pi$ of physical fields is not necessarily correlated with the ghost parity.
\end{remark}

We now describe the space of BV-\/functionals. First, we denote by~$\overline{H}^n(\pi_\BV)$ the space of equivalence classes (modulo the image of horizontal differential~$\overline{\Id}$) of top\/-\/degree horizontal forms on~$\overline{J^\infty_\bpi}(\Pi\widehat{\bpi})$; in other words, elements $\cH\in\overline{H}^n(\pi_\BV)$ are integral functionals possibly depending on the entire set of BV-\/variables~$\bq$ and~$\bq^\d$ and their derivatives~$\bq_\sigma$,\ $\bq^\d_\tau$ of arbitrarily high but finite orders (by convention, $\bq_\varnothing\equiv\bq$, $\bq^\d_\varnothing\equiv\bq^\d$ and $0\leqslant|\sigma|,|\tau|<\infty$). The full BV-\/action $S_\BV=S+\ldots$ for a given gauge\/-\/invariant model is a standard example of such integral functional.

Every equivalence class $\cH = \int h\bigl(x,[\bq],[\bq^\d]\bigr)\,\dvol\in\ol{H}^n(\pi_\BV)$ of the highest horizontal cohomology group for the vector bundle $\pi_\BV$ 
determines a map $\cH\colon\Gamma(\pi_\BV)\to\Bbbk$ (which we let be either~$\mathbb{R}$ or, possibly,~$\mathbb{C}$). Namely, for any $s\in\Gamma(\pi_\BV)$ we set
\begin{equation}\label{eq:BB}
\cH \colon s \mapsto \cH(s) = \int_M \jet(s)^* h\bigl(x,[\bq],[\bq^\d]\bigr)\,\dvol
\in \Bbbk.
\end{equation}
The BV-\/action~$S_\BV$ is a typical mapping that takes sections to numbers. 
We employ such integral functions from $\ol{H}^n(\pi_\BV)$ as building blocks in the construction of a larger set of `admissible' functionals which also map 
$\Gamma(\pi_\BV)$~to the field~$\Bbbk$. In agreement with the ideology of~\cite{topical,redbook}, let us view sections~$s \in \Gamma(\pi_\BV)$ as ``points'' and cohomology classes from~$\ol{H}^n(\pi_\BV)$ as particular examples of $\Bbbk$-\/valued ``functions'' defined at every ``point'' of the space of 
sections.\footnote{Note that all highest horizontal cohomology group elements are such ``functions'', but not all ``functions'' are integral functionals.} Next, we introduce the multiplicative structure $F\mathbin{{\otimes}_\Bbbk} G \mapsto F\cdot G$ pointwise by the rule
\[
(F\cdot G)(s)\stackrel{\text{def}}{{}={}}F(s)\stackrel{\Bbbk}{{}\cdot{}}G(s), \quad s\in\Gamma(\pi_\BV),
\]
for two integral functionals $F, G \in \ol{H}^n(\pi_\BV)$.

By taking (formal) sums of products of arbitrary finite numbers of integral functionals from $\ol{H}^n(\pi_\BV)$, i.e., by viewing such functionals as building blocks, we generate the linear subspace,
\begin{align}\label{eq:functionals}
\ol{\mathfrak{M}}^n(\pi_\BV) = \bigoplus_{i=1}^{+\infty}\mathop{{\bigotimes}_{\lefteqn{\Bbbk}}}\nolimits^i \, \ol{H}^n(\pi_\BV) \subseteq 
\mathfrak{M}^n(\pi_\BV) = \text{Map}_\Bbbk\bigl(\Gamma(\pi_\BV)\to\Bbbk\bigr),
\end{align}
in the space $\mathfrak{M}^n(\pi_\BV)$ of \emph{all} maps which assign a number to every section. We do not claim that the subspace $\ol{\mathfrak{M}}^n(\pi_\BV)$ of local functionals exhausts the entire set of existing mappings; we emphasize that in what follows we describe the construction of the Schouten bracket and 
BV-\/Laplacian on that linear subspace $\ol{\mathfrak{M}}^n(\pi_\BV) \subseteq \mathfrak{M}^n(\pi_\BV)$ (that is, not just on the set of equivalence classes 
$\ol{H}^n(\pi_\BV)$, still not for arbitrary maps in~$\mathfrak{M}^n(\pi_\BV)$). The reason why we do so is that integration by parts makes sense for each constituent term from~$\ol{H}^n(\pi_\BV)$ of a composite functional 
from~$\ol{\mathfrak{M}}^n(\pi_\BV)$.

\begin{remark}\label{remark:pi-products}
By constructing the space of maps $\ol{\mathfrak{M}}^n(\pi_\BV)$ in this fashion, we have not completely exited the realm of integral functionals on jet spaces. To see this, take two integral functionals $F = \int f(x,[\bq],[\bq^\d])\,\dif^nx$ and $G = \int g(x',[\bq'],[\bq'^\d])\,\dif^nx'$ on $\pi_\BV$. Then we may write $F\cdot G$ as $F\cdot G = \int  f(x,[\bq],[\bq^\d])\,g(x',[\bq'],[\bq'^\d])\,\dif^nx\wedge\dif^nx'$, which shows that $F\cdot G \in \ol{H}^{2n}(\pi_\BV\times\pi_\BV)$, which is the top level cohomology of the jet space of the product bundle $\pi_\BV\times\pi_\BV$. Thus, $F\cdot G$ is an integral functional with respect to the bundle $\pi_\BV\times\pi_\BV$, and we see that the second term $\ol{H}^n(\pi_\BV)\otimes_\Bbbk\ol{H}^n(\pi_\BV)$ in the direct sum in~\eqref{eq:functionals} is a subspace of $\ol{H}^{2n}(\pi_\BV\times\pi_\BV)$. Continuing this line of reasoning to products of any number of factors we find
\[
\ol{\mathfrak{M}}^n(\pi_\BV) \subseteq \bigoplus_{i=1}^{+\infty} \ol{H}^{i\cdot n}(\underbrace{\pi_\BV\times\cdots\times\pi_\BV}_\text{$i$ copies}).
\]
(Note that this is a strict 
inclusion; for example, $\ol{H}^{2n}(\pi_\BV\times\pi_\BV)$ contains functionals such as $\int q_{xx'}\,\dif^nx\wedge\dif^nx'$ while $\ol{H}^n(\pi_\BV)\otimes_\Bbbk\ol{H}^n(\pi_\BV)$ does not.) Thus, although composite but homogeneous functionals from $\ol{\mathfrak{M}}^n(\pi_\BV)$ (i.e., functionals of the form $F = F_1\cdots F_i$ for some $i$ and $F_k \in \ol{H}^n(\pi)$ for $1\leq k \leq i$) are not integral with respect to the bundle $\pi_\BV$, they \emph{are} integral with respect to a bundle, namely $\underbrace{\pi_\BV\times\cdots\times\pi_\BV}_\text{$i$ copies}$.
\end{remark}

\section{The Schouten bracket}\label{SecSchouten}
Consider first the space $\ol{H}^n(\pi_\BV)$ of integral functionals, that is, equivalence classes of highest\/-\/degree horizontal forms modulo exact forms in the image of $\hdif = \pi_\BV^*(\dif_{\mathrm{dR}(M)})$. We are dealing with a geometry in which densities of integral functionals can depend on the 
(anti)\/fields and all the available generations of (anti)\/ghosts, and their derivatives. Let us remember that the ``anti\/-\/objects'' $q^\d_\alpha$~are introduced as the 
duals of the BV-\/vector coordinates~$q^\beta$ so that 
$\langle \delta q^\alpha, \delta q^\d_\beta\rangle = +1\cdot\delta^\alpha_\beta$
and 
$\langle\delta q^\d_\beta,\delta q^\alpha\rangle = -1\cdot\delta^\alpha_\beta$ are the $\Bbbk$-\/valued couplings of the dual bases\footnote{At this point, the appearance of the minus sign -- whenever the variations are swapped -- is logical because the Cartan differentials must anticommute under the wedge product. On the other hand, this sign change will re-appear in sections~\ref{sec:conventional laplacian}--\ref{SecMaster}, taking the shape of a choice of the dual bases in fibres of the bundle $\bpi\mathbin{{\times}_M}\Pi\wh\bpi$: in $\Pi(\wh{\Pi\wh\bpi})$, the basis of dual vectors to a frame in $\Pi\wh\bpi$ is \emph{minus} the initially taken basis in a fibre of $\bpi$.} (obviously, $\delta^\alpha_\beta$ is the Kronecker symbol, not a variation).

To explain the minus sign in the second coupling,
let us recall that 
the variations $\delta(\cdot)$ of jet variables are 
shorthand notation for a projection of the vertical Cartan differential, which acts~--along the fibres of~$\pi_\infty$~-- onto the horizontal cohomology with respect to the horizontal differential~$\overline{\Id}$.
Consequently, the geometry is $\mathbb{Z}_2$-\/graded by the parity\/-\/reversion operator~$\Pi$, which is introduced by hands to produce Manin's odd neighbours (c.f.~\cite{Voronov} and~\cite{ArthemyPrague,ArthemyAlgebroids}), and it is also $\mathbb{Z}$-\/graded by the degree of Cartan forms (this degree equals zero for the action~$S\in\ol{H}^n(\pi)$, BV-\/action~$S_\BV \in \ol{H}^n(\pi_\BV)$, and generally, for all elements of~$\ol{H}^n(\pi_\BV)$ -- but equals one for the variation~$\delta S_\BV$, etc.). On top of that, the initial 
bundle~$\pi$ itself can be graded by a suitable ring (usually, it is~$\mathbb{Z}_2$ again so that $\pi=\bigl(\pi^{\overline{0}}|\pi^{\overline{1}}\bigr)$ and these parities of the fibre variables in~$\pi$ are~$\varepsilon_\alpha \in \mathbb{Z}_2$). For the sake of clarity we assume onwards that the bundle $\pi$ of physical fields is purely even; nevertheless, the variation symbols $\delta$ anticommute between themselves, independently from the ghost\/-\/parity~$\Pi$ that distinguishes 
between $q^\alpha$ and $q^\d_\beta$. We let $\delta q^\alpha \wedge \delta q^\d_\beta$ be the odd\/-\/parity symplectic form on $\Jet(\bpi)\mathbin{{\times}_M}\Jet(\Pi\wh\bpi)$.

\begin{definition}\label{def:Schouten}
Let $F, G \in \ol{H}^n(\pi_\BV)$ be two integral functionals depending, generally speaking, on the entire collection of BV-variables and their derivatives along the base $M$ of arbitrary high but finite order. The \term{variational Schouten bracket} of the building blocks $F$ and $G$ is
\begin{equation}\label{EqDefSchouten}
\schouten{F,G} = \langle\overrightarrow{\delta F}\wedge\overleftarrow{\delta G}\rangle,
\end{equation}
where the arrows indicate the direction in which, with due attention to the 
$\mathbb{Z}_2$-\/parity brought in by~$\Pi$, the variations of 
BV-\/variables~$\bq$ and~$\bq^\d$ are transported. After this transportation they are coupled, as indicated by the angular brackets $\langle\,\cdot\,\rangle$.
\end{definition}

\begin{remark}
This construction of the odd Poisson bracket $\schouten{\,,} \colon \ol{H}^n(\pi_\BV) \times \ol{H}^n(\pi_\BV) \to \ol{H}^n(\pi_\BV)$ is standard 
(see~\cite[Ch.\,8]{ArthemyBook} for its extension to a noncommutative setup of cyclic\/-\/invariant words); let us now focus on the mechanism of the ordered coupling between the initially separate geometries for~$F$ and~$G$. Namely, as was explained in Remark~\ref{remark:pi-products}, each of these two factors is defined over its own copy of the bundle $\hJet\bpi(\Pi\wh\bpi) \to \Jet(\bpi) \to M$ with all canonically conjugate pairs of BV-\/variables~$\bq$ and~$\bq^\d$ and their derivatives~$\bq_\sigma$ and~$\bq^\d_\tau$ for coordinates along the fibres over points of the base manifold~$M$. In each of the two copies we vary, that is, we take the vertical differential~$\Id_{\cC}$ and integrate by parts.\footnote{The value of the Schouten bracket $\schouten{F,G}$ does not depend on the choice of densities in both factors because the variations of exact forms are equal to zero by our earlier convention of the absence of boundary terms.} Since $F$ and $G$ are defined over separate copies of $\pi_\BV$, the integration by parts on the one does not feel the other:
\begin{align*}
   \lshad&F,G\rshad
     = \langle\overrightarrow{\delta F}\wedge\overleftarrow{\delta G}\rangle 
     = \left\langle
         \int\overrightarrow{\dif_\Ct}\big(f(x,[\bq],[\bq^\d])\big)\,\dif^nx \wedge
         \int\overleftarrow{\dif_\Ct}\big(g(y,[\bq'],[\bq'^\d])\big)\,\dif^ny
      \right\rangle \\
   &= \left\langle \iint
         \left.\left(\frac{\overrightarrow{\delta f}}{\delta q^\alpha}\,\delta q^\alpha + \frac{\overrightarrow{\delta f}}{\delta q^\d_\alpha}\,\delta q^\d_\alpha\right)\right|_{(x,[\bq],[\bq^\d])}
         \wedge
         \left.\left(\delta q^\alpha\,\frac{\overleftarrow{\delta g}}{\delta q^\alpha} + \delta q^\d_\alpha\,\frac{\overleftarrow{\delta g}}{\delta q^\d_\alpha}\right)\right|_{(y,[\bq'],[\bq'^\d])}
         \dif^nx\,\dif^ny
      \right\rangle.
\end{align*}
   
Now begins the merging of the two bundles: we first identify their bases~$M$, then the fibres of~$\pi$, and finally we identify 
both copies of 
the BV-\/bundle~$\pi_\BV$.
It is the coupling $\langle\,,\rangle$ of covectors and vectors which provides the identification of points in the mechanism of merging. Indeed,
the pairs of coupling\/-\/dual Cartan differentials from $\overrightarrow{\delta F}$ and $\overleftarrow{\delta G}$, whenever attached to \emph{different} points of the space $\hJet\bpi(\Pi\wh\bpi)$, couple to zero, whereas the same coupling yields $\pm 1 \in \Bbbk$ if these points coincide (this is identical to the coupling $T_y^*N\times T_x N\to\mathbb{R}$ for any smooth manifold~$N$). 
   
This scenario can be recognized also as Dirac's $\delta$-\/distribution $\delta(x-y)$, in which $y$~runs over the base~$M$ and which integrates away\footnote{The convention that the Schouten bracket of two integral functionals remains an integral functional determines a natural choice of the $\Bbbk$-valued coupling $\langle\,,\rangle$ by using the variables~$\bq^\d$ rather than 
a 
choice of a $\ol\Lambda^n(\pi_\BV)$-valued coupling by using the 
metric\/-\/dependent Hodge structure~$*$ on~$M$ and the conjugate variables~$\bq^*$, see, e.g.,~\cite{CF}.} the $n$~differentials $\Id^n y$ in one of the two integral functionals, $F=\int f\,\dvol$ and $G=\int g\,\dif^n y$. Hence we conclude that the contributions from~$F$ and~$G$ to their bracket $\schouten{F,G}$ are evaluated at the (infinite jet~$\jet(s)$ of the) same section~$s\in\Gamma(\pi_\BV)$.
\end{remark}

\begin{corollary}\label{thm:Schouten}
In coordinates, the Schouten bracket $\schouten{F,G}$ of two functionals $F = \int f\,\dvol$,\ $G = \int g\,\dif^ny \in \ol{H}^n(\pi_\BV)$ is the integral functional given by the formula
\[
\schouten{F,G} = \int\left(\rdifv[f]{q^\alpha}\cdot\ldifv[g]{q^\d_\alpha} - \rdifv[f]{q^\d_\alpha}\cdot\ldifv[g]{q^\alpha}\right)\,\dvol,
\]
where the arguments of all the differential functions coincide and are equal 
to $\theta^\infty=\bigl(x,[\bq],[\bq^\d]\bigr)\in \hJet\bpi(\Pi\wh\bpi)$.
\end{corollary}

Finally, we extend the Schouten bracket from the space $\ol{H}^n(\pi_\BV)$ of building blocks~\eqref{eq:BB} to the linear subspace $\ol{\mathfrak{M}}^n(\pi_\BV)$ generated in~$\mathfrak{M}^n(\pi_\BV)$ by using elements from~$\ol{H}^n(\pi_\BV)$ under the multiplication in $\Bbbk$ for their values at sections $s \in \Gamma(\pi_\BV)$. We take the following theorem as the recursive definition.

\begin{theorem}\label{thm:LeibnizSchouten}
The geometric construction of $\schouten{\,,}$ in Definition~\textup{\ref{def:Schouten}} extends to the \emph{antibracket} $\schouten{\,,}$ on $\ol{\mathfrak{M}}^n(\pi_\BV)$ such that the Leibniz rule
\begin{align*}
\schouten{F,G\cdot H}(s) &= 
\big(\schouten{F,G}\cdot H\big)(s) + (-)^{(\gh(F)-1)\cdot\gh(G)}
 \big(G\cdot\schouten{F,H}\big)(s) \nonumber\\
&= \schouten{F,G}(s)\cdot H(s) + (-)^{(\gh(F)-1)\cdot\gh(G)}
 G(s)\cdot\schouten{F,H}(s)
\end{align*}
holds for any $F,G,H\in\ol{\mathfrak{M}}^n(\pi_\BV)$ and all $s\in\Gamma(\pi_\BV)$. The antibracket~$\schouten{\,,}$ on~$\ol{\mathfrak{M}}^n(\pi_\BV)$ is bi\/-\/linear over $\Bbbk$,
shifted\/-\/graded skew\/-\/symmetric,
\[
\schouten{F, G} = -(-)^{(\gh(F)-1)(\gh(G)-1)}\schouten{G,F},
\]
 and satisfies the 
shifted\/-\/graded Jacobi identity\textup{:}~\footnote{For a proof of these last three standard properties (namely, bi-linearity, etc.) of the variational Schouten bracket in the (non)\/commutative setup see, e.g.,~\cite[Ch.\,8]{ArthemyBook}, \cite{Lorentz12} and~\cite{protaras}.}
\begin{multline}
   0 = (-)^{(\gh(F)-1)(\gh(H)-1)}\bschouten{F,\schouten{G,H}}
    + (-)^{(\gh(F)-1)(\gh(G)-1)}\bschouten{G,\schouten{H,F}} \\
    + (-)^{(\gh(G)-1)(\gh(H)-1)}\bschouten{H,\schouten{F,G}}. \label{EqJacobiSchouten}
\end{multline}
\end{theorem}
\begin{proof}
For simplicity, we assume that $F \in \ol{H}^n(\pi_\BV)$ is an integral functional. Since the bracket is shifted-graded skew-symmetric, any reasoning that applies to the right slot of the bracket will also hold for the left slot, so no generality is lost. We also assume, again without loss of generality, that $G$ is also an integral functional, and that $H$ is composite but homogeneous (i.e. $H = H_1\cdots H_{i}$, where the factors are elements of $\ol{H}^n(\pi_\BV)$). By realizing that $H = H_1\cdot H_2 \cdots H_i$ is also a product of one integral functional with several others, one may continue the following process inductively to reduce the bracket of $F$ and $G\cdot H$ to an expression which contains only brackets with integral functionals on $\pi_\BV$ as arguments.

Following Remark~\ref{remark:pi-products}, we consider $G\cdot H$ to be an integral functional with respect to the product bundle
\[
\pi_\BV\times\underbrace{\pi_\BV\times\ldots\times\pi_\BV}_{\text{$i$ copies}}=
\stackrel{(0)}{\pi_\BV'}\times\stackrel{(1)}{\pi_\BV''}\times\ldots\times\stackrel{(i)}{\pi_\BV''},
\]
of~$1+i$ copies of the BV-\/bundle~$\pi_\BV$. The total space of this product of bundles is endowed with the Cartan differential
\[
\Id_\cC=\stackrel{(0)}{\Id_\cC'}+\stackrel{(1)}{\Id_\cC''}+\ldots+\stackrel{(i)}{\Id_\cC''},
\]
where the superscripts indicate from which copy of $\pi_\BV$ the operator comes. Note that the integration by parts, whenever one transforms~$\stackrel{(a)}{\Id_\cC'}\mapsto\stackrel{(a)}{\delta'}$ or~$\stackrel{(a)}{\Id_\cC''}\mapsto\stackrel{(a)}{\delta''}$ inside the $a$th building block~$\overline{H}^n(\pi)$ within~$G$ or~$H$, does not produce any effect on the other $i$ factors because their geometries are entirely independent from the geometry of~$\stackrel{(a)}{\pi_\BV}$. Consequently, all total derivatives from~$\stackrel{(a)}{\pi_\BV}$ falling outside~$\overline{H}^n\bigl(\stackrel{(a)}{\pi_\BV}\bigr)$ evaluate to zero, whence
\[
\overleftarrow{\delta}(H)=\sum_{a=1}^{i'}H_1\cdot\ldots\cdot\overleftarrow{\delta}(H_a)\cdot\ldots\cdot H_{i},
\]
where the resulting variations $\delta q^\alpha$ and $\delta q^\d_\alpha$ are to be transported all the way to the left through $H_1\cdots H_{a-1}$, as indicated by the arrow on top of the Euler operator $\delta$. Summarizing, we have
\[
\overleftarrow{\delta (G \cdot H)} = \overleftarrow{\delta' G}\cdot H + G\cdot\overleftarrow{\delta'' H}.
\]
Then the Schouten bracket is given by equation~\eqref{EqDefSchouten} (suppressing the suffix $(s)$ which indicates that everything is evaluated on a section $s$):
\begin{align*}
   \schouten{F, G \cdot H} 
   &= \langle\overrightarrow{\delta F}\wedge\overleftarrow{\delta (G \cdot H)}\rangle
    = \langle\overrightarrow{\delta F}\wedge \overleftarrow{\delta' G}\cdot H\rangle
        + \langle\overrightarrow{\delta F}\wedge G\cdot\overleftarrow{\delta'' H}\rangle
\intertext{In the right term, the differentials have to be transported to the left before the coupling can happen. In order to achieve this we swap $G$ with $\overleftarrow{\delta'' H}$, obtaining a minus sign $(-)^{\gh(G)\gh(H)}$:}
   &= \langle\overrightarrow{\delta F}\wedge \overleftarrow{\delta' G}\rangle\cdot H
        + (-)^{\gh(G)\gh(H)}\langle\overrightarrow{\delta F}\wedge \overleftarrow{\delta'' H}\rangle\cdot G
\intertext{Swapping the factors in the right term, we obtain:}
   &= \langle\overrightarrow{\delta F}\wedge \overleftarrow{\delta' G}\rangle\cdot H
        + (-)^{\gh(G)\gh(H)}(-)^{(\gh(F)+\gh(H)-1)\gh(G)} G\cdot\langle\overrightarrow{\delta F}\wedge \overleftarrow{\delta'' H}\rangle \\
   &= \langle\overrightarrow{\delta F}\wedge \overleftarrow{\delta' G}\rangle\cdot H
        + (-)^{(\gh(F)-1)\gh(G)} G\cdot\langle\overrightarrow{\delta F}\wedge \overleftarrow{\delta'' H}\rangle \\
   &= \schouten{F,G}\cdot H + (-)^{(\gh(F)-1)\gh(G)}G\cdot\schouten{F,H}.
\end{align*}
This proves the claim.
\end{proof}

\section{The 
BV-\/Laplacian}\label{sec:BV}
The BV-\/Laplacian~$\lap_\BV$ is a specific linear operator on the linear subspace $\ol{\mathfrak{M}}^n(\pi_\BV) \subseteq \mathfrak{M}^n(\pi_\BV)$ of those functionals $\cH\colon\Gamma(\pi_\BV)\to\Bbbk$ which are assembled by using (sums of products of) the integral functionals from $\ol{H}^n(\pi_\BV)$. As in the previous section, we start with the definition of~$\lap_\BV$ on the space~$\ol{H}^n(\pi_\BV)$ of elementary building blocks.

\begin{definition}[provisional, c.f.~\S\ref{sec:conventional laplacian}]\label{def:Laplacian}
The \term{BV-\/Laplacian}~$\lap_\BV \colon \ol{H}^n(\pi_\BV) \to \ol{H}^n(\pi_\BV)$ on the space of integral functionals -- possibly, depending on the entire collection of BV-variables and on their derivatives up to arbitrarily high order -- is the linear mapping
\begin{equation}\label{eq:DefLaplacian}
\lap_\BV\colon \mathcal{H} \mapsto \tfrac12\langle\overleftarrow\delta *\overleftarrow\delta(\mathcal{H})\rangle, \quad \cH\in\ol{H}^n(\pi_\BV),
\end{equation}
where $\langle\:\rangle$~denotes the ordered $\Bbbk$-valued coupling $\langle\delta q^\alpha\wedge\delta q^\d_\beta\rangle = \delta^\alpha_\beta = -\langle\delta q^\alpha\wedge\delta q^\d_\beta\rangle$ of Cartan dif\-fe\-ren\-ti\-als for the pairs $q^\alpha\leftrightarrow q^\d_\alpha$ of canonically conjugate variables provided~by $\pi_\BV$, and where the operator~$*=(-)^{\gh(\cdot)-1}$ is 
such that $*(\delta q^\alpha)=-\delta q^\alpha$ and~$*(\delta q^\d_\beta) = \delta q^\d_\beta$.
\end{definition}

\begin{corollary}
In coordinates, the BV-\/Laplacian~$\lap_\BV$ acts on an integral functional 
$\mathcal{H} = \int h\bigl(x$,\ $[\bq]$,\ $[\bq^\d]\bigr)\,\dvol \in \ol{H}^n(\pi_\BV)$ by the formula
\begin{equation}\label{eq:laplacian-in-coordinates}
\lap_\BV(\mathcal{H}) = \int \left.\left(\ldifv{q^\alpha}\left(\ldifv[h]{q^\d_\alpha}\right)\right)\right|_{(x,[\bq],[\bq^\d])}\,\dvol,
\end{equation}
which thus yields an integral functional again\textup{:} 
$\lap_\BV(\mathcal{H}) \in \ol{H}^n(\pi_\BV)$.
\end{corollary}

\begin{proof}
By the construction of BV-\/variables~$\bq$ and~$\bq^\d$ we have that~$\overleftarrow{\delta}=\overleftarrow{\delta_\bq}+\overleftarrow{\delta_{\bq^\d}}$. Note that the effect of the operator $*$ is to turn $\delta_\bq$ into $-\delta_\bq$, while it maps $\delta_{\bq^\d}$ to itself. Then
\begin{align*}
   \frac12\langle\overleftarrow\delta *\overleftarrow\delta(\mathcal{H})\rangle
    = \frac12\Big\langle(\overleftarrow{\delta_\bq} + \overleftarrow{\delta_{\bq^\d}})(-\overleftarrow{\delta_\bq}+\overleftarrow{\delta_{\bq^\d}})(\mathcal{H})\Big\rangle
    = \frac12\Big\langle(-\overleftarrow{\delta_\bq}^2 + \overleftarrow{\delta_\bq}\overleftarrow{\delta_{\bq^\d}} - \overleftarrow{\delta_{\bq^\d}}\overleftarrow{\delta_\bq} - \overleftarrow{\delta_{\bq^\d}}^2)(\mathcal{H})\Big\rangle
\end{align*}
Since $\delta_\bq$ and $\delta_{\bq^\d}$ are differentials, the leftmost and rightmost terms vanish; moreover, they anticommute, so the second and third term collapse onto each other:\footnote{In our conventions, the left arrow on $\overleftarrow\delta$ means that the variations $\delta q^\d_\alpha$ are to be pushed to the left. As a consequence of this, the operator $\overleftarrow\delta/\delta q^\d_\alpha$ acts from the left; in the literature this is usually denoted with an arrow pointing towards the \emph{right}. Therefore the partial derivatives $\overrightarrow\partial/\partial q^\d_{\alpha,\sigma}$ that occur in $\overleftarrow\delta/\delta q^\d_\alpha$ have reversed arrows.} 
\begin{align} \label{eq:Laplacian-halfway} 
   &= \Big\langle \overleftarrow{\delta_\bq}\big(\overleftarrow{\delta_{\bq^\d}}(\mathcal{H})\big)\Big\rangle 
    = \left\langle\overleftarrow{\delta_\bq}\left(
   \int \delta q^\d_\alpha\,\ldifv[h]{q^\d_\alpha}\,\dif^nx
   \right)\right\rangle 
    = \left\langle\int\delta q^\beta \, (-)^{|\tau|}D_\tau\left(\delta q^\d_\alpha\,\frac{\overrightarrow\partial}{\partial q^\beta_\tau}\ldifv[h]{q^\d_\alpha}\right)\,\dif^nx\right\rangle.
\end{align}
Some of the derivatives will fall on $\delta q^\d_\alpha$, while the others fall on what stands to the right of $\delta q^\d_\alpha$. Let us recall, however, that the $\Bbbk$-\/valued coupling from~$\bpi_\infty^*(\Pi\wh{\bpi}_\infty)$ for the variations of canonically conjugate variables is such that
\begin{align*}
\couple{\dif_\Ct q^\beta_\tau,\dif_\Ct q^\d_{\alpha,\sigma}} &= \phantom{+}(-)^{\gh(q^\alpha)}\delta^\alpha_\beta\cdot\delta^\varnothing_{\sigma\cup\tau}, \\
\couple{\dif_\Ct q^\d_{\alpha,\sigma},\dif_\Ct q^\beta_\tau} &= -(-)^{\gh(q^\alpha)}\delta^\alpha_\beta\cdot\delta^\varnothing_{\sigma\cup\tau},
\end{align*}
where the second Kronecker delta\/-\/symbol is nonzero only if both 
multi\/-\/indexes~$\sigma$ and~$\tau$ are empty. This definition selects only those terms in which \emph{all} the derivatives fall on the coefficients of the vertical two\/-\/form. Thus, we obtain
\begin{align*}
   \frac12\langle\overleftarrow\delta *\overleftarrow\delta(\mathcal{H})\rangle
   &= \int\langle\delta q^\beta \wedge \delta q^\d_\alpha\rangle\,
   \left.\left(\ldifv{q^\beta}\left(\ldifv[h]{q^\d_\alpha}\right)\right)\right|_{(x,[\bq],[\bq^\d])}\,\dvol,
\end{align*}
which is the well\/-\/known expression~\eqref{eq:laplacian-in-coordinates} from the literature.
\end{proof}

\begin{example}\label{ex:YM}
Take a compact, semisimple Lie group~$G$ with Lie algebra~$\mathfrak{g}$ and consider the corresponding Yang-Mills theory. Write $A^a_i$ for the (coordinate expression of) the gauge potential $A$ -- a lower index $i$ because $A$ is a one-form on the base manifold (i.e., a covector), and an upper index $a$ because $A$ is a vector in the Lie algebra $\mathfrak{g}$ of the Lie group $G$. Defining the field strength $\mathcal{F}$ by $\mathcal{F}^a_{ij} = \partial_i A^a_j - \partial_j A^a_i + f^a_{bc}A^b_i A^c_j$ where $f^a_{bc}$ are the structure constants of $\mathfrak{g}$, the Yang-Mills action is
\[
   S_0 = \frac14\int\mathcal{F}^a_{ij}\mathcal{F}^{a,ij}\,\dif^nx,
\]
and the full BV-action $S_\BV$ is
\[
   S_\text{YM} = S_0
   + \int A_a^{i\dagger}(D_i\gamma^a + f_{bc}^a A_i^b\gamma^c) \,\dif^4x
    - \frac12\int f_{ab}^c\gamma^a\gamma^b\gamma^\dagger_c\,\dif^4x.
\]
Let us calculate the BV-Laplacian of this functional. As a consequence of equation~\eqref{eq:laplacian-in-coordinates}, the only terms which survive in $\lap_\BV(S_\text{YM})$ are those which contain both $A$ and $A^\d$, or both $\gamma$ and $\gamma^\d$. Therefore,
\begin{align*}
\lap_\BV(S_\text{YM})
 &= \int\left(
    \ldifv{A_j^d}\ldifv{A^{j\d}_d}(f_{bc}^a A_a^{i\dagger}A_i^b\gamma^c)
   - \frac12\ldifv{\gamma^\d_d}\ldifv{\gamma^d}(f^c_{ab}\gamma^a\gamma^b\gamma^\dagger_c)
   \right)\dif^4x \\
 &= \int\left(
   \ldifv{A_j^d}(f_{bc}^d A_j^b\gamma^c)
   - \frac12\ldifv{\gamma^\d_d}(f^c_{db}\gamma^b\gamma^\dagger_c - f^c_{ad}\gamma^a\gamma^\dagger_c)
   \right)\dif^4x \\
 &= \int\left(
   f_{dc}^d\gamma^c - \frac12(f^d_{db}\gamma^b + f^d_{ad}\gamma^a)
   \right)\dif^4x 
  = 0.
\end{align*}
Since the BV-action $S_\text{YM}$ is by construction such that $\schouten{S_\text{YM},S_\text{YM}}$ is zero, it follows that $S_\text{YM}$ satisfies the quantum master equation~\eqref{eq:QME} tautologically -- both sides are, by independent calculations, equal to zero.
\end{example}

\begin{example}\label{ex:CF}
Consider the nonlinear Poisson sigma model introduced in~\cite{CF}. Since its fields are not all purely even, we would have to generalize all of our reasoning so far to a $\BBZ_2$-graded setup -- which is, as noted before, tedious but straightforward. A calculation of $\lap_\BV(S_\text{CF})$ of the BV-action $S_\text{CF}$ of this model would, up to minor differences in conventions and notations, proceed just as it does in the paper itself, in section 3.2 -- except that no infinite constants or delta functions appear.
\end{example}

\begin{proposition}\label{thm:LaplacianDifferential}
The linear operator $\lap_\BV \colon \ol{H}^n(\pi_\BV) \to \ol{H}^n(\pi_\BV)$ is a differential\textup{,}
\[
\left(\lap_\BV\right)^2\colon\ol{H}^n(\pi_\BV)\to0.
\]
\end{proposition}

\begin{proof}
Using equation~\eqref{eq:laplacian-in-coordinates} twice, we may write a repeated application of~$\lap_\BV$ on~$\mathcal{H} = \int h\bigl(x$,\ $[\bq]$,\ $\bq^\d]\bigr)\,\dvol\in\ol{H}^n(\pi_\BV)$ as
\begin{equation}\label{eq:laplacian^2}
\lap_\BV^2(\mathcal{H}) = \int \ldifv{q^\alpha}\ldifv{q^\d_\alpha}\ldifv{q^\beta}\ldifv[h]{q^\d_\beta}\,\dvol.
\end{equation}
Consider the rightmost variational derivative,
\[
\ldifv{q^\d_\beta} = \frac{\overrightarrow\partial}{\partial q^\d_\beta} + \sum_{\tau\neq\varnothing}D_\tau\circ\frac{\overrightarrow\partial}{\partial q^\d_{\beta,\tau}}.
\]
To the immediate left of this in equation~\eqref{eq:laplacian^2} stands the next variational derivative,~$\overleftarrow{\delta}/\delta{q^\beta}$. Since $\ldifv{q^\beta}\circ D_\tau = 0$ for any~$\beta$ and~$\tau$, all of the terms in the sum above disappear, and the rightmost variational derivative~$\overleftarrow{\delta}/\delta{q^\d_\beta}$ becomes just~${\overrightarrow\partial}/{\partial q^\d_\beta}$ in equation~\eqref{eq:laplacian^2}. A~similar process happens with the second and the third variational derivatives. As to the last, leftmost one, all terms of this variational derivative that contain total derivatives do not contribute to the functional, because the integral over a total derivative is zero according to our convention on no boundary terms. Summarizing, we have
\begin{align*}
\lap_\BV(\mathcal{H}) 
 &= \int \frac{\overrightarrow\partial}{\partial q^\alpha}\frac{\overrightarrow\partial}{\partial q^\d_\alpha}\frac{\overrightarrow\partial}{\partial q^\beta}\frac{\overrightarrow{\partial h}}{\partial q^\d_\beta}\,\dvol.
\intertext{Since the middle two partial derivatives in this expression commute, this becomes}
{}&= \int \frac{\overrightarrow\partial}{\partial q^\alpha}\frac{\overrightarrow\partial}{\partial q^\beta}\frac{\overrightarrow\partial}{\partial q^\d_\alpha}\frac{\overrightarrow{\partial}}{\partial q^\d_\beta}(h)\,\dvol.
\end{align*}
This is the composition 
of an expression which is symmetric in~$\alpha$ and~$\beta$ (the left two partial derivatives) and an expression which is antisymmetric in~$\alpha$ and~$\beta$ (the right two partial derivatives). Therefore it is zero.
\end{proof}

As we did before with the Schouten bracket, we now extend the BV-\/Laplacian from the space~$\ol{H}^n(\pi_\BV)$ of building blocks~\eqref{eq:BB} to the linear subspace $\ol{\mathfrak{M}}^n(\pi_\BV)$ generated in~$\mathfrak{M}^n(\pi_\BV)=\text{Map}\,(\Gamma(\pi_\BV)\to\Bbbk)$ by elements from~$\ol{H}^n(\pi_\BV)$ under multiplication in~$\Bbbk$ of their values at sections~$s \in \Gamma(\pi_\BV)$.
 
\begin{theorem}\label{thm:LeibnizLaplacian}
The geometric construction of the BV-\/Laplacian 
in Definition~\textup{\ref{def:Laplacian}} extends to the linear operator~$\lap_\BV$ on the space $\ol{\mathfrak{M}}^n(\pi_\BV)$ such that
\begin{align}\label{eq:laplacian-on-product}
\lap_\BV(F\cdot G)(s)
&= (\lap_\BV(F)\cdot G)(s) + (-)^{\gh(F)}\schouten{F,G}(s) + 
 (-)^{\gh(F)}(F\cdot\lap_\BV(G))(s) \notag\\
{}&= (\lap_\BV F)(s)\cdot G(s) + (-)^{\gh(F)}\schouten{F,G}(s) + 
 (-)^{\gh(F)}F(s)\cdot(\lap_\BV G)(s),
\end{align}
for any $F,G\in\ol{\mathfrak{M}}^n(\pi_\BV)$ and all~$s\in\Gamma(\pi_\BV)$\textup{;} the antibracket~$\schouten{\,,}$ measures the deviation of $\lap_\BV$ from being a derivation.
\end{theorem}

\begin{proof}
Without loss of generality let us assume that the local functionals~$F=F_1\cdot\ldots\cdot F_{i'}$ and~$G=G_1\cdot\ldots\cdot G_{i''}$ are homogeneous terms in~\eqref{eq:BB}, consisting of~$i'$ and~$i''$ building blocks from~$\overline{H}^n(\pi_\BV)$, respectively. Following Remark~\ref{remark:pi-products}, we then consider the product,
\[
\underbrace{\pi_\BV\times\ldots\times\pi_\BV}_{\text{$i'$ copies}}\times
\underbrace{\pi_\BV\times\ldots\times\pi_\BV}_{\text{$i''$ copies}}{}
={}
\stackrel{(1)}{\pi_\BV'}\times\ldots\times\stackrel{(i')}{\pi_\BV'}\times
\stackrel{(1)}{\pi_\BV''}\times\ldots\times\stackrel{(i'')}{\pi_\BV''},
\]
of~$i'+i''$ copies of the BV-\/bundle~$\pi_\BV$. The total space of this product of bundles is endowed with the Cartan differential
\[
\Id_\cC = \Id_\cC' + \Id_\cC''
\]
where $\Id_\cC'$ is the Cartan differentials on the $i'$ copies of $\pi_\BV$ of $F$, and $\Id_\cC''$ is the Cartan differentials on the $i''$ copies of $\pi_\BV$ of $G$.

We start from the left hand side of formula~\eqref{eq:Laplacian-halfway}:
\begin{align*}
\lap_\BV(F\cdot G)
&= \Big\langle\overleftarrow{\delta_\bq}\big(\overleftarrow{\delta_{\bq^\d}}(F\cdot G)\big)\Big\rangle 
 = \Big\langle \Big(\overleftarrow{\delta_\bq'} + \overleftarrow{\delta_\bq''}\Big)\Big(\overleftarrow{\delta_{\bq^\d}'} + \overleftarrow{\delta_{\bq^\d}''}\Big)(F\cdot G)\Big\rangle \\
&= \Big\langle \Big(\overleftarrow{\delta_\bq'} + \overleftarrow{\delta_\bq''}\Big)\Big(\overleftarrow{\delta_{\bq^\d}'}F\cdot G + (-)^{\gh(F)} F\cdot \overleftarrow{\delta_{\bq^\d}''}G\Big)\Big\rangle \\
&= \Big\langle
\overleftarrow{\delta_\bq'}\overleftarrow{\delta_{\bq^\d}'}F\cdot G 
- \overleftarrow{\delta_{\bq^\d}'}F\cdot \overleftarrow{\delta_\bq''}G
+(-)^{\gh(F)} \overleftarrow{\delta_\bq'}F\cdot \overleftarrow{\delta_{\bq^\d}''}G
+(-)^{\gh(F)} F\cdot \overleftarrow{\delta_\bq''}\overleftarrow{\delta_{\bq^\d}''}G
\Big\rangle.
\end{align*}
The second term carries a minus sign because the anticommuting differentials $\overleftarrow{\delta_{\bq^\d}'}$ and $\overleftarrow{\delta_\bq''}$ have swapped positions in it. Reversing the direction of the arrow of $\overleftarrow{\delta_{\bq^\d}'}F$ in this term (i.e., the odd\/-\/parity variation~$\delta\bq^\d$ is transported to the right instead of left), it gains an extra sign~$(-)^{\gh(F)-1}$. A similar arrow reversion~$\overleftarrow{\delta_\bq}\mapsto\overrightarrow{\delta_\bq}$ in the third term does not produce a minus sign. Then the third and second terms combine to the Schouten bracket, while we recognize the first and the last one as the Laplacian acting on $F$ and $G$, respectively:
\[
\lap_\BV(F\cdot G) = 
\Big\langle\overleftarrow{\delta'_\bq}\overleftarrow{\delta'_{\bq^\d}}F\cdot G
+(-)^{\gh(F)}\left(
 \overrightarrow{\delta'_\bq}F\wedge\overleftarrow{\delta''_{\bq^\d}}G 
 +\overrightarrow{\delta'_{\bq^\d}}F\wedge\overleftarrow{\delta''_{\bq}}G 
\right)
+(-)^{\gh(F)}F\cdot\overleftarrow{\delta''_\bq}\overleftarrow{\delta''_{\bq^\d}}G\Big\rangle.
\]
Finally, we couple the Cartan differentials of the BV-\/variables and the assertion follows.
\end{proof}

Contrary to the conventional BV-Laplacian (see Proposition~\ref{ThLapSchouten} on page~\pageref{ThLapSchouten}), the BV-Laplacian $\lap_\BV$ defined here does \emph{not} always satisfy the equation $\lap_\BV(\schouten{F,G}) = \schouten{\lap_\BV F,G} + (-)^{\gh(F)-1}[\![F$,\ $\lap_\BV G]\!]$. The reason is that if we were to proceed with the calculation of $\lap_\BV(\schouten{F,G})$ as in the proof of Proposition~\ref{ThLapSchouten} (see below), it is no longer permissible to swap the symbols $\delta/\delta q^\alpha$ -- which here stand for the \emph{variational} derivative -- with each other. Therefore the four remaining terms in equation~\eqref{eq:conventional-laplacian-schouten} do not cancel. To see this explicitly, take the following counterexample.\footnote{We
warn the reader that all claims in Counterexample~\ref{counterexample} about any equalities between functionals (specifically, for $\Delta_{\BV}F$ and~$\Delta_{\BV}G$ or for $\lshad F,G\rshad$ and~$\Delta\bigl(\lshad F,G\rshad\bigr)$, etc.) should be viewed as the classical parable about a cage which contains an elephant but carries an inscription ``\textsc{Buffalo}''~--- one may not trust own eyes. We refer to Remark~\ref{RemSynonyms} in section~\ref{sec:conventional laplacian} and also to Example~\ref{ExCounter2} on p.~\pageref{ExCounter2} in which we explicitly calculate the objects $\Delta F$ and~$\Delta G$ and~$\Delta\bigl(\lshad F,G\rshad\bigr)$, confirming the validity of equality~\eqref{EqNotHolds}.}

\begin{counterexample}\label{counterexample}
Let the base and fibre both be one-dimensional, and set
\[ 
   F = \int q^\d q q_{xx}\,\dif x
   \quad\text{and}\quad
   G = \int q^\d_{xx}\cos q\,\dif x.
\]
Let $f$ and $g$ be the two integrands. We calculate
\begin{align}\label{eq:counterexample}
   \rdifv[f]{q}
      &= q^\d q_{xx} + D_x^2(q^\d q)
       = q^\d q_{xx} + q^\d_{xx}q + 2q^\d_x q_x + q^\d q_{xx}
       = 2q^\d q_{xx} + q^\d_{xx}q + 2q^\d_x q_x, \\
   \rdifv[f]{q^\d} &= qq_{xx}, \qquad\qquad
   \ldifv[g]{q} = -q^\d_{xx}\sin q, \nonumber\\
   \ldifv[g]{q^\d}
      &= D_x^2(\cos q)
       = D_x(-q_x\sin q)
       = -q_{xx}\sin q - q_x^2\cos q.\nonumber
\end{align}
(Note that since all four variational derivatives contain at most one parity-odd $q^\d$ or its derivatives, the directions of the arrows do not actually matter -- i.e., switching their direction does not result in minus signs.) Consider $\lap_\BV(\schouten{F,G})$. As was noted in Proposition~\ref{thm:LaplacianDifferential}, we may write $\lap_\BV(\mathcal{H}) = \int(\partial/\partial q)\circ(\partial/\partial q^\d)(\mathcal{H})$ for any $\mathcal{H} \in \ol{H}^n(\pi_\BV)$; therefore, only terms in which $\schouten{F,G}$ carries at least one $q^\d$ without derivatives with respect to the base space survive. This implies that in the first term of the bracket, $\overrightarrow{\delta f}/\delta q\cdot\overleftarrow{\delta g}/\delta q^\d$, the second and third term of the right hand side of~\eqref{eq:counterexample} do not contribute, so we need not take them into account:
\[
   \schouten{F,G} = \int\big(2q^\d q_{xx}(-q_{xx}\sin q - q_x^2\cos q) + \cdots
    - qq_{xx}\cdot(-q^\d_{xx}\sin q)\big)\,\dif x,
\]
where the dots indicate the omitted terms. For the same reason, the last term also does not contribute.

We calculate
\begin{align*}
   \lap_\BV(\schouten{F,G})
   &= -2\int\frac\partial{\partial q}\frac\partial{\partial q^\d}
      (q^\d q_{xx}^2\sin q + q^\d q_{xx}q_x^2\cos q)\,\dif x \\
   &= -2\int\frac\partial{\partial q}(q_{xx}^2\sin q + q_{xx}q_x^2\cos q)\,\dif x \\
   &= -2\int(q_{xx}^2\cos q - q_{xx}q_x^2\sin q)\,\dif x,
\end{align*}
whose integrand is not cohomogically trivial (as may be seen by calculating its variational derivative, which gives nonzero).

Now $\lap_\BV F = \int q_{xx}\,\dif x \cong 0$, so $\schouten{\lap_\BV F,G} = 0$. On the other hand, $g$ has no $q^\d$ without any $x$-derivatives in it so $\lap_\BV G = 0$, so $\schouten{F, \lap_\BV G} = 0$ as well. In conclusion,
\begin{equation}\label{EqNotHolds}
   \schouten{\lap_\BV F,G} + (-)^{\gh(F)-1}\schouten{F, \lap_\BV G} = 0 \neq \lap_\BV(\schouten{F,G}).
\end{equation}
\end{counterexample}

This counterexample shows that we have reached the limits of a jet-space approach to the BV-geometry via Vinogradov's $\Ct$-spectral sequence $E^{p,q}_i$ (specifically, the upper line $E^{n,q}_1$ of its first term such that $E^{n,0}_1 = \ol{H}^n(\pi_\BV)$). Let us note, however, that until now we operated with the objects $\dif_\Ct q^\alpha$ and $\delta q^\alpha$ or $\dif_\Ct q^\d_\alpha$ or $\delta q^\d_\alpha$ by viewing them as the Cartan differentials of dependent variables (essentially, the De Rham differentials).

We now pass from the space of infinite jets of sections in the BV-bundle $\pi_\BV$ to the space of sections $\Gamma(\bpi\mathbin{{\times}_M}\Pi\wh\bpi)$, that is, to the (anti)fields. Retaining the full arsenal of already known objects and structures, we let the approach become slightly more functional analytic. For example, we recall that both $\bpi$ and $\Pi\wh\bpi$ are vector bundles over the base manifold $M$ and that a section's \emph{variation} $\delta s = (\delta s^\alpha, \delta s^\d_\beta)$ of $s \in \Gamma(\bpi\mathbin{{\times}_M}\Pi\wh\bpi)$ itself is a section of the bundle $T(\bpi\mathbin{{\times}_M}\Pi\wh\bpi)$.

\section{The functional definition of BV-Laplacian}\label{sec:conventional laplacian}
\noindent%
To approach the self-regularization in a genuine definition of the Batalin\/--\/Vilkovisky Laplacian, we analyze first
the basic geometry of variation of functionals. Namely, let us study the interrelation of bundles in the course
of integration by parts, the implications of the locality postulate, and a rigorous construction of the iterated
functional derivatives.

The core of difficulties which became manifest in Counterexample~\ref{counterexample} is that a use of the vector bundle 
$\pi_{\text{BV}}$ in the construction of the space $\ov{\gM}(\pi_{\BV})$, which contains formal sums of products
$F_1\cdot\ldots\cdot F_k$ of integral functionals $F_i\in\ov{H}^n(\pi_{\BV})$, is insufficient to grasp the full
geometry of the calculus of variations. Indeed, several important
identities
combining the Schouten bracket with the BV-Laplacian do not hold; such identities involve higher-order variational
derivatives, but those need to be regularized (or proclaimed permutable) whenever one inspects a response of a functional to
a shift of its argument along different test sections, i.\,e., in several independent directions. To circumvent 
the difficulties, we shall enlarge the space of functionals in such a way that there is enough room to store the
information about test shifts. Paradoxally, this generalization not only encodes properly the variations of (anti)fields
but also communicates to the functionals a kind of ``memory'' of the way in which they were obtained from primary
functionals (essentially, from the observables). The effect of the functionals' memory is manifest through the existence
of synonymic mappings $\Gamma(\pi_{\BV})\to\Bbbk$ which yield equal numbers for every given section yet which belong
to non-isomorphic spaces and behave differently under multiplication by using the Schouten bracket or under application
of the BV-Laplacian.

This approach resolves the obstructions in a problem of intrinsic regularization of the Schouten bracket and the
BV-\/Laplacian; the newly defined objects match in all standard ways (see Theorems~\ref{ThLapSchouten} 
and~\ref{ThBVDifferential} below).

For the sake of brevity, we denote by $[\bu]$ a differential dependence on the unknown variables (specifically to the BV-\/setup, a dependence on~$[\bq]$ 
and~$[\bq^\dagger]$).

\begin{example}\label{ExOneStep}
In particular, we shall explain why the conventional formula
\begin{equation}\label{EqConventFD}
\left.\frac{\Id}{\Id\veps}\right|_{\veps=0}F(s+\veps\cdot\overleftarrow{\delta s})=\int_M\Id\bx\,\de s(\bx)\cdot
\left.\frac{\overleftarrow{\delta f}(\bx,[\bu])}{\de\bu}\right|_{j^{\infty}_{\bx}(s)}
\end{equation}
for calculation of a functional's response to a test shift of its argument is a \textit{consequence} of the definition~---
but not itself a definition.
Containing a single functional derivative of $F=\int f(\bx,[\bu])\,\Id\bx\in\ov{H}^n(\pi_{\BV})$, formula~\eqref{EqConventFD}
is especially instructive: we claim that a simplicity with which a correct expression is obtained by just one step is
misleading; it hides a longer reasoning of which the right-hand side of~\eqref{EqConventFD} is an implication. Indeed,
let us notice that the left-hand side of~\eqref{EqConventFD} refers to \textit{three} bundles (namely, $\pi_{\BV}$ for the integral functional $F\in\ov{H}^n(\pi_{\BV})$, the vector bundle $\bpi\mathbin{{\times}_M}\Pi\widehat{\bpi}$ for a background section~$s$, and the tangent bundle $T(\bpi\mathbin{{\times}_M}\Pi\widehat{\bpi})$ for the shift section~$\de s$).\footnote{Because the fibre $V_x\oplus\Pi V_x^\dagger$ over~$x\in M$ in the bundle $\bpi\mathbin{{\times}_M}\Pi\widehat{\bpi}$ is a vector space, its tangent space at a point~$s(x)$ is isomorphic to the fibre. Still we emphasize that there is a distinction between~$s$ and~$\delta s$ (they are sections of different bundles). We note further that $\de s\in\Gamma\bigl(T(\bpi\mathbin{{\times}_M}\Pi\widehat{\bpi})\bigr)$ is postulated to be independent of $s\in\Gamma(\bpi\mathbin{{\times}_M}\Pi\widehat{\bpi})$, i.e., the shift of~$s$ at~$x$ is the same at all values~$s(x)\in V_x\oplus\Pi V_x^\dagger$ so that the notation~$\delta s(x)$ makes sense (a rigorous formula would be $\delta s\bigl(s(x)\bigr)\in T_{s(x)}\bigl(V_x\oplus\Pi V_x^\dagger\bigr)$. The independence of test shifts of the background field manifests the translation invariance of the measure in Feynman's path integral (see section~\ref{SecMaster}).}
However, these domains of definition merge to the integration manifold $M$ in the right-hand side. Therefore, from~\eqref{EqConventFD} it remains
unclear whether the variational derivative $\overleftarrow{\delta f}/\de\bu$ refers to one (which would be false) or two (true!) copies 
of $M$. (In fact, we have that 
$\overleftarrow{\delta f}(\bx,[\bu])/\de\bu=\sum\limits_{|\sigma|\geqslant0}\left.\left(-\frac{\Id}{\Id\bby}\right)^{\sigma}\right|_{\bby=\bx}
\bigl(\vec{\dd}f(\bx,[\bu])/\dd\bu_{\sigma}\bigr)$.) We argue that the multiplication $\cdot$ in the integrand 
of~\eqref{EqConventFD} is a result of the coupling in the defining equality
\[
\left.\frac{\Id}{\Id\veps}\right|_{\veps=0}F(s+\veps\cdot\overleftarrow{\delta s})=\int_M\Id\bx\int_M\Id\bby
\left\langle\de s(\bby),\left.\left(-\frac{\Id}{\Id\bby}\right)^{\sigma}
\left(\frac{\vec{\dd}f(\bx,[\bu])}{\dd\bu_{\sigma}}\right)\right|_{j^{\infty}_{\bx}(s)}\right\rangle.
\]
The mechanism of congruence $\bby=\bx$ and the geometric self-regularization of Dirac's delta\/-\/distribution, 
which is provided by that mechanism, appeals to the coupling 
\[
\int_M a(\bby){\bigr|}_{\bby=\bx}\cdot b(\bx)\,\Id\bx=\int_M\Id\bx\int_M\Id\bby\,\langle a(\bby),b(\bx)\rangle
\]
of (co)vector fields $\langle a|$ and $|\,b\rangle$ on $M$. Thus, as a by\/-\/product we clarify the derivation of the Euler\/--\/Lagrange equation $\de F(s_0)=0$ upon the stationary points~$s_0$ of an action functional~$F$. 

However, for identities containing several variational derivatives it becomes crucial that the differential
of an integral functional's density is always referred 
to the ``private'' copy of the base $M\ni\bx$ owned in $\ov{H}^n(\pi_{\BV})$ by that functional, whereas the
integrations by parts are always performed with respect to the individual domains of definition $M\ni\bby_1,\ldots,M\ni\bby_k$ 
of the variations. A restriction to the diagonal $\bx=\bby_1=\ldots=\bby_k$ stems from the coupling of those auxiliary vector
bundles' sections $\de s_1,\ldots,\de s_k$ with the $k$-th order variational derivative; the latter is a section of the dual
vector bundle~$T^*(\bpi\mathbin{{\times}_M}\Pi\widehat{\bpi})$ 
whenever any $k-1$ variations are viewed as parameters in the coupling.
\end{example}

Let us remember that the initial bundle $\pi$ and the bundles $\bpi$ of fields and $\widehat{\bpi}$ of antifields
(such that $\ov{J^{\infty}_{\pi}}(\pi_{\BV})=J^{\infty}(\bpi)\mathbin{{\times}_M}J^{\infty}(\Pi\widehat{\bpi})$)
are \textit{vector} bundles over $M$. This means that the fibre over each point $x\in M$ in $\bpi\mathbin{{\times}_M}\Pi\widehat{\bpi}$ is 
a vector space $V_{\bx}\oplus\Pi V_{\bx}^\dagger$ endowed with a $\Bbbk$-linear structure; the sum of two sections
$s\in\Gamma(\bpi\mathbin{{\times}_M}\Pi\widehat{\bpi})$ and $\veps\cdot\de s\in\Gamma\bigl(T(\bpi\mathbin{{\times}_M}\Pi\widehat{\bpi})\bigr)$ is defined pointwise on $M$:
we have that $(s+\veps\cdot\de s)(\bx)=s(\bx)+\veps\cdot\de s(\bx)\in V_{\bx}\oplus\Pi V_{\bx}^\dagger$ (here we use the vector space structure of the fibre
$V_{\bx}\oplus\Pi V_{\bx}^\dagger$, also identifying it with its own tangent space at~$s(x)$ so that the sum of sections makes sense).
This is the locality postulate which is brought into the model by hand. 
(Note that in what follows we avoid an identification of $\Gamma\bigl(T(\bpi\mathbin{{\times}_M}\Pi\widehat{\bpi})\bigr)$ with $\Gamma(\bpi\mathbin{{\times}_M}\Pi\widehat{\bpi})$ even if we deal with the fibrewise\/-\/constant shifts~$\delta s$ for which $\delta s\bigl(s_1(x)\bigr)=\delta s\bigl(s_2(x)\bigr)\mathrel{{=}{:}}\delta s(x)$ for any $s_1,s_2\in\Gamma(\bpi\mathbin{{\times}_M}\Pi\widehat{\bpi})$ at~$x\in M$.) 
The locality postulate is a mechanism which communicates the value of $\veps\cdot\de s$ at $\bby\in M$ or the values at $\bby\mathrel{{:}{=}}\bx$ of
derivatives of $\veps\cdot\de s$ along its domain of definition to the geometry that carries the section $s$ and handles its value at~$\bx\in M$.

We now transform the postulate of locality (which stemmed from the initial hypothesis that the bundles' fibres are vector
spaces) and the multiplication $\cdot$ of the objects' coefficients to the mechanism of congruence $\bby=\bx$ of attachment
points for (co)vector fields $\langle\,|(\bby)$ or $|\,\rangle(\bby)$ and $|\,\rangle(\bx)$ or $\langle\,|(\bx)$ in the coupling
$\langle\,,\,\rangle\colon T_{s(y)}\Pi V_{\bby}^\dagger\times T_{s(x)}V_{\bx}\to\Bbbk$ and $\langle\,,\,\rangle\colon T_{s(y)}V_{\bby}\times T_{s(x)}\Pi V^\dagger_{\bx}\to\Bbbk$ (respectively, we have $\langle\,,\,\rangle\colon T^*_{s(y)}V_{\bby}\times T_{s(x)}V_{\bx}\to\Bbbk$ and $\langle\,,\,\rangle\colon T_{s(y)}V_{\bby}\times T^*_{s(x)}V_{\bx}\to\Bbbk$ up to the introduction of ghost parity).
This will be the key element in a self\/-\/regularization procedure for $\lshad\ ,\,\rshad$ and $\Delta_{\BV}$ when,
at the second step of the reasoning, we explicitly use the duality between the halves of sections for
$\bpi\mathbin{{\times}_M}\Pi\widehat{\bpi}$.
The operational definitions of Schouten bracket and BV-\/Laplacian amount to  surgery algorithms for reconfigurations of
such couplings (see Definition~\ref{DefBVLap} below).
Currently, we deal with iterated functional derivatives in full generality, not actually referring to the composite
structure of $\bpi\mathbin{{\times}_M}\Pi\widehat{\bpi}$ and of the associated tangent bundle.

Let us extend the intrinsic geometry of each building block in construction~\eqref{eq:BB} of the composite functionals space
$\ov{\gM}^n(\pi_{\BV})$. Leaving the multiplication of integral functionals intact, we replace the bundle $\pi_{\BV}$ for
$j^{\infty}(s)$ in each $\ov{H}^n(\pi_{\BV})$ by the product of the BV\/-\/bundle~$\pi_{\BV}$ with $k$~copies $T\pi_{\BV}\times\ldots\times T\pi_{\BV}$ of
its tangent bundle;
the future sections of such product bundle will be composed by
$s\in\Gamma(\bpi\mathbin{{\times}_M}\Pi\widehat{\bpi})$ and $k$~variations $\de s_1,\dots,\de s_k$ which are constant along each fibre in~$\pi_{\BV}$ 
(here~$k\ge0$). 
We introduce the space\footnote{By construction, elements of $\ov{\gN}^n(\pi_{\BV})$ will determine linear maps with respect
to each variation $\veps_j\cdot\de s_j(y_j)$ in the limit~$\veps_j\to0$.}
\begin{multline*}
\ov{\gN}^n(\pi_{\BV})=\bigoplus\limits_{i=1}^{+\infty}\bigotimes\nolimits_{\Bbbk}^i\bigoplus\limits_{k=0}^{+\infty}
\ov{H}^{n(1+k)}(\pi_{\BV}\times\underbrace{T\pi_{\BV}\times\ldots\times T\pi_{\BV}}_{k\text{ copies}})\subseteq{}\\
{}\subseteq\text{Map}\bigl(
  \Gamma(\pi_{\BV})\times\smash{
   \overbrace{\Gamma(T\pi_{\BV})\times\ldots\times\Gamma(T\pi_{\BV}) }^%
{k\text{ variations}}  }\to\Bbbk\bigr).
\end{multline*}
Note that the product of bundles in the above formula is not a Whitney sum over the common base $M$; conversely, 
each variation $\de s_j$ brings its own copy of the base into the geometry. An analogous technique was mentioned
in Remark~\ref{remark:pi-products} on p.~\pageref{remark:pi-products}; that approach will be used again --~at a later stage~-- in the functional definition
of the Schouten bracket that combines two functionals from $\ov{\gN}^n(\pi_{\BV})\supsetneq\ov{\gM}^n(\pi_{\BV})$
to one functional. However, we now analyze the mechanism of merging the bases $M\times M\times\ldots\times M$ in the 
product of bundles inside each factor $F_j\in\ov{H}^{n(1+k)}(\pi_{\BV}\times T\pi_{\BV}\times\ldots\times T\pi_{\BV})$
of a homogeneous term $F_1\cdot\ldots\cdot F_i$ of~$\ov{\gN}^n(\pi_{\BV})$.

Consider an integral functional $F=\int f(\bx,[\bu])\,\Id\bx\in\ov{H}^n(\pi_{\BV})\subset\ov{\gM}^n(\pi_{\BV})$
and let $s\in\Gamma(\bpi\mathbin{{\times}_M}\Pi\widehat{\bpi})$. To inspect a response of the functional 
$F\colon s\mapsto F(s)\in\Bbbk$ 
to an infinitesimal shift of its argument $s$ by $k$ test sections $\de s_1,\ldots,\de s_k$ of the vector bundle
$T(\bpi\mathbin{{\times}_M}\Pi\widehat{\bpi})$, we take 
\[
\left.\frac{\Id}{\Id\veps'}\right|_{\veps'=0}
\left.\frac{\Id}{\Id\veps''}\right|_{\veps''=0}\cdots
\left.\frac{\Id}{\Id\veps^{(k)}}\right|_{\veps^{(k)}=0}
F(s+\veps'\cdot\overleftarrow{\delta s}_1+\ldots+\veps^{(k)}\cdot\overleftarrow{\delta s}_k).
\]
By using the chain rule we have that
\begin{multline*}
\left.\frac{\Id}{\Id\veps'}\right|_{\veps'=0}\cdots
\left.\frac{\Id}{\Id\veps^{(k)}}\right|_{\veps^{(k)}=0}
\left.\int_M\Id\bx\,f(\bx,[\bu])
{\bigr|}_{j^{\infty}_{\bx}(s)+\veps^{(i)}\cdot j^{\infty}_{\bby_i}(\overleftarrow{\delta s}_i)}\right|_{\bby_i:=\bx}=\\=
\int_M\Id\bx\left.\Biggl\{\sum_{j=1}^k\sum_{\substack{|\sigma_j|\ge0,\\i_j=1,\dots,
  {\rank(\bpi\mathbin{{\times}_M}\Pi\widehat{\bpi})}}}
\left.\left(\frac{\dd}{\dd\bby_1}\right)^{\sigma_1}\right|_{\bby_1}(\de s_1^{i_1})\cdot\ldots\cdot
\left.\left(\frac{\dd}{\dd\bby_k}\right)^{\sigma_k}\right|_{\bby_k}(\de s_k^{i_k})\cdot\right.\\
\cdot\left.
\left.\frac{\vec{\dd}^kf(\bx,[\bu])}{\dd u_{\sigma_1}^{i_1}\dots\dd u_{\sigma_k}^{i_k}}\right|_%
{j^{\infty}_{\bx}(s)}\Biggr\}\right|_{\bby_j\mathrel{{:}{=}}\bx}.
\end{multline*}
We emphasize that the derivatives of a density $f$ for $F$ refer to the base $M\ni\bx$ of that integral functional.
At the same time, the derivatives of fibrewise\/-\/constant test shifts~$\de s_j$ are taken with respect to their own domains of definition~$M\ni y_j$ so that
it is the numbers from~$\Bbbk$ which are communicated to the operation of multiplication $\cdot$ at the (very last)
moment of restriction to the diagonal~$\bby_j=\bx$.

To substantiate this mechanism, let us deal not with the \textit{coefficients} $\de s_j^{i_j}(\bby_j)$ of sections $\de s_j$
and not with the multiplication $\cdot$ in $\Bbbk$, but let us reformulate the picture in terms of (co)frame fields 
over~$M$ and their couplings. Recall the standard geometry of dual bases located in the fibres of a Whitney sum of two
vector bundles over points $\bx,\bby\in M$ of their base: if $\bx=\bby$, the coupling of two representatives in the fibres
works out an appropriate number from $\Bbbk$; otherwise, if $\bx\neq\bby$, the coupling is automatically zero.

Note that at each $\bby_i\in M$ the value 
$\bigl({\dd}/{\dd\bby_i}\bigr)^{\sigma_i}(\de s_i)(\bby_i)$ 
is an element of the vector space 
$T_{s(y_i)}\bigl(V_{\bby_i}\oplus\Pi V_{\bby_i}^\dagger\bigr)$
and the value
$\vec{\dd}f(\bx,[\bu])/\dd\bu_{\sigma}{\bigr|}_{j^{\infty}_{\bx}(s)}$ at $\bx\in M$ belongs to the dual vector space
$T^*_{s(x)}\bigl(V_{\bx}\oplus\Pi V_{\bx}^\dagger\bigr)$ over~$\bx$. 
We repeat that the derivatives
$\bigl({\dd}/{\dd\bby_i}\bigr)^{\sigma_i}$ along the base manifold $M\ni\bby_i$ act on the coefficients of sections but
not on the frame fields. This implies that the linear part of a response of $F$ to a shift of its argument away from $s$
is equal to
\begin{multline*}
\left.\frac{\Id}{\Id\veps'}\right|_{\veps'=0}\cdots\left.\frac{\Id}{\Id\veps^{(k)}}\right|_{\veps^{(k)}=0}
F(s+\veps'\cdot\overleftarrow{\delta s}_1+\ldots+\veps^{(k)}\cdot\overleftarrow{\delta s}_k)=\\=
\int_M\Id\bx\int_M\Id\bby_1\cdots\int_M\Id\bby_k\left\langle
\begin{matrix}
\bigl({\dd}/{\dd\bby_1}\bigr)^{\sigma_1}(\de s_1)(\bby_1),&\vec{\dd}/\dd\bu_{\sigma_1}\\
\vdots&\vdots\\
\bigl({\dd}/{\dd\bby_k}\bigr)^{\sigma_k}(\de s_k)(\bby_k),&\vec{\dd}/\dd\bu_{\sigma_k}
\end{matrix}
\bigl(f(\bx,[\bu])\bigr){\bigr|}_{j^{\infty}_{\bx}(s)}\right\rangle.
\end{multline*}
Each line within $\langle\,,\,\rangle$ contains a coupling of the dual bases; it forces the points $\bby_i$ and $\bx$ to
coincide at the moment of evaluation at $(1+k)$ sections (but not earlier). Integrating by parts, we obtain
(by using the agreement on the absence of boundary terms)
\begin{align}
{}\cong&{}\int_M\Id\bx\int_M\Id\bby_1\cdots\int_M\Id\bby_k
 \label{EqDeltaViaCoupling}\\
{}&\qquad
\left\langle\de s_1(\bby_1)\ldots\de s_k(\bby_k),
\left.\left(-\frac{\dd}{\dd\bby_1}\right)^{\sigma_1}\right|_{\bby_1}\cdots
\left.\left(-\frac{\dd}{\dd\bby_k}\right)^{\sigma_k}\right|_{\bby_k}
\left.\left(\frac{\overrightarrow{\dd^k}f(\bx,[\bu])}{\dd\bu_{\sigma_1}\dots\dd\bu_{\sigma_k}}\right)
\right|_{j^{\infty}_{\bx}(s)}\right\rangle.\notag
\end{align}
Clearly, before its evaluation at a concrete section $s\in\Gamma(\bpi\mathbin{{\times}_M}\Pi\widehat{\bpi})$ and at given shifts
$\de s_1,\dots,\de s_k$, the object we are dealing with exists as an element of top-cohomology group
$\ov{H}^{n(1+k)}(\pi_{\BV}\times T\pi_{\BV}\times\ldots\times T\pi_{\BV})$. 
In particular, various identities involving
variational derivatives (e.\,g., see~\eqref{EqLapSchouten} below) are equalities between the maps taking $s$ 
and all its variations $\de s_i$ to a number from~$\Bbbk$. In conclusion, all the couplings which force the congruence
of points $\bby_i$ and $\bx\in M$ are never performed in the course of verification of such identities within
$\ov{\gN}^n(\pi_{\BV})$. Moreover, whenever an extra derivative, corresponding to $\de s_{k+1}(\bby_{k+1})$, falls on an
object whose intrinsic geometry already encodes $k$ shifts of $s$, the $(k+1)$-th order derivative of the initial 
functional's density $f$ is still referred to its own base $M\ni\bx$; it does not feel the presence of any test shifts.
Consequently, each functional's derivative adds a new copy of its domain of definition and inserts an extra tangent bundle~$T\pi_{\BV}$ to the product $\pi_{\BV}\times T\pi_{\BV}\times\ldots\times T\pi_{\BV}$.

\begin{remark}
A one\/-\/step reduction to the diagonal $\bby=\bx$ has been illustrated in Example~\ref{ExOneStep} on p.~\pageref{ExOneStep}.
Another case --~with multiple functional derivatives~-- will be considered in Example~\ref{ExCounter2} in order to show 
an instance of the surgery technique for reconfigurations of the couplings; the geometry of that model is specific to the BV-setup of (anti)fields.
\end{remark}

\begin{remark}\label{RemSynonyms}
There exist integral functionals $F\in\ov{H}^n(\pi_{\BV})$ and
$G\in\ov{H}^{n+2n}(\pi_{\BV}\times\pi_{\BV}\times\pi_{\BV})$ and there are
(specially chosen, see below) test shifts $\de s_1$ and $\de s_2$ such that $F(s)=G(s,\de s_1,\de s_2)\in\Bbbk$ for every
$s\in\Gamma(\bpi\mathbin{{\times}_M}\Pi\widehat{\bpi})$ but such that $F$ and $G$ behave differently with respect to the Schouten
bracket or BV-Laplacian. For example, one can find $H\in\ov{H}^n(\pi_{\BV})$ such that $\lshad F,H\rshad\ne\lshad G,H\rshad$
as elements of $\ov{\gN}^n(\pi_{\BV})$ and $\lshad F,H\rshad(s)\ne\lshad G,H\rshad(s,\de s_1,\de s_2)\in\Bbbk$ or such that
$\Delta_{\BV}(F)\ne\Delta_{\BV}(G)$ and $(\Delta_{\BV}F)(s)\ne(\Delta_{\BV}G)(s,\de s_1,\de s_2)$.

Namely, let $\de s_1=(\de s_1^1,\dots,\de s_1^N;0,\dots,0)\in\Gamma\bigl(T(\bpi\mathbin{{\times}_M}\Pi\widehat{\bpi})\bigr)$, here
$N=\rank\bpi$, and let $\de s_2=(0,\dots,0;\de s_{2,1}^\dagger,\dots,\de s_{2,N}^\dagger)$ such that
$\de s^i_1\cdot\de s_{2,i}^\dagger=1$ for every $i$ (a possibility to have that normalization independent from a choice of local
coordinates is guaranteed by the composite geometry of fibres in the vector bundle $\bpi\mathbin{{\times}_M}\Pi\widehat{\bpi}$). Let
$f(\bx,[\bq],[\bq^\dagger])$ be a differential function on the jet space. Then the functionals
\begin{align*}
F&\mathrel{{:}{=}}\int\Id\bx\,\sum_{i=1}^N\left\{\de s_1^i(\bx)\cdot\de s_{2,i}^\dagger(\bx)\cdot
\frac{\lde^2 f(\bx,[\bq],[\bq^\dagger])}{\de\bq^i\de\bq_i^\dagger}\right\}\\
\intertext{and (see~\eqref{EqDeltaViaCoupling} above}
G&\mathrel{{:}{=}}\int\Id\bx\int\Id\bby\int\Id\bz\left\{\sum_{i=1}^N\sum_{|\sigma_1|,|\sigma_2|\geqslant0}
\de s_1^i(\bby)\cdot\de s_{2,i}^\dagger(\bz)\cdot\right.\\
&{}\qquad{}\cdot
\left.\left.\left(-\frac{\Id}{\Id\bby}\right)^{\sigma_1}\right|_{\bby}
\left.\left(-\frac{\Id}{\Id\bz}\right)^{\sigma_2}\right|_{\bz}
\frac{\overrightarrow{\dd^2}f(\bx,[\bq],[\bq^\dagger])}{\dd\bq^i_{\sigma_1}\dd\bq^\dagger_{i,\sigma_2}}\right\}\cdot
\de(\bby-\bx)\de(\bz-\bx)
\end{align*}
are indistinguishable as mappings to $\Bbbk$ for every $s\in\Gamma(\bpi\mathbin{{\times}_M}\Pi\widehat{\bpi})$, yet the difference in
the behaviour of $F$ and $G$ under basic operations of differential calculus provides a resolution to the obstructions
in Counterexample~\ref{counterexample}, c.\,f.\ Example~\ref{ExCounter2}. We see that unequal elements from $\ov{\gM}^n(\pi_{\BV})$
and $\ov{\gN}^n(\pi_{\BV})$ can determine synonymic maps of the space of sections. At the same time, elements 
$F\in\ov{\gM}^n(\pi_{\BV})$ could be viewed as primary functionals; their descendants $G\in\ov{\gN}^n(\pi_{\BV})$ retain 
a kind of memory of the way in which they have been derived from some primary objects in 
$\ov{H}^n(\pi_{\BV})\subseteq\ov{\gM}^n(\pi_{\BV})$. Moreover, such memory governs the behaviour of descendants in the
course of analytic operations (see Example~\ref{ExCounter2} below for a nontrivial synonym~$\Delta G$ of the zero functional).
\end{remark}

The above reasoning was unspecific to the BV\/-\/geometry; with elementary modifications in notation it could be applied, e.g., to a derivation of the Euler\/--\/Lagrange equations upon sections of a given vector bundle (see Example~\ref{ExOneStep}).
We now use the initial assumption that the BV-bundle $\bpi\mathbin{{\times}_M}\Pi\wh\bpi$ has a very special structure of its fibres. Namely, the dependent variables $\bq$ and $\bq^\d$ are grouped in two halves, each set's cardinality being equal to $m_0 + m_1 + \cdots + m_\lambda =: N$. The basic vectors $\vec{e}^{\,\d,1},\dots,\vec{e}^{\,\d,N}$ in the second summand of the fibre $V_x\oplus\Pi V^\d_x\simeq
T_{s(x)}\bigl(V_x\oplus\Pi V^\d_x\bigr)$ over each~$x\in M$ are dual to the respective vectors $\vec{e}_1,\dots,\vec{e}_N$, spanning $V_x$, under the $\Bbbk$-valued coupling $\langle \vec{e}_\alpha(x), \vec{e}^{\,\d,\beta}(y)\rangle = \delta_\alpha^\beta\cdot\delta(x-y)$. (Note that the coupling changes the ghost parity by convention; the vectors $\vec{e}^{\,\d,1},\dots,\vec{e}^{\,\d,N}$ themselves carry odd parity by the construction of $\Pi\wh\bpi$, but the Kronecker delta in the coupling is postulated to be ghost-parity even.) For any (smooth) choice of a frame field of basic vectors $\vec{e}_\alpha$ and $\vec{e}^{\,\d,\beta}$ over $M$, a local section $s \in \Gamma(\bpi\mathbin{{\times}_M}\Pi\wh\bpi)$ of the vector bundle at hand is then
\[
s = s^1\cdot\vec{e}_1 + \cdots + s^N\cdot\vec{e}_N + s^\d_1\cdot\vec{e}^{\,\d,1} + \cdots + s^\d_N\cdot\vec{e}^{\,\d,N},
\]
for some $\Bbbk$-valued functions $s^\alpha$ and $s^\d_\beta$ on $M$; their smoothness class is governed by an external agreement which is brought in by hand.

Continuing this line of reasoning, we see that the test shifts $\delta s \in \Gamma\bigl(T(\bpi\mathbin{{\times}_M}\Pi\wh\bpi)\bigr)$ of the functionals' arguments inherit this splitting: $\delta s = (\delta s^\alpha, \delta s^\d_\beta)$. Furthermore, one could either employ the shifts $\delta s_1 = (\delta s^\alpha, 0)$ or $\delta s_2 = (0, \delta s^\d_\beta)$ of pure ghost parities $(+)$ and $(-)$, respectively. Alternatively, one can split a given section $\delta s = (\delta s^\alpha, \delta s^\d_\beta) = (\delta s^\alpha, 0) + (0, \delta s^\d_\beta)$ and then use these two terms at different stages.
Let us point out a convenient normalization $\delta s^\alpha(x)\cdot\delta s^\d_\alpha(x) \equiv 1$ at all $x \in M$ and a fixed index $\alpha = 1,\dots,N$; this constraint is coordinate-independent owing to the duality of $V_x$ and $V^\d_x$. (The right hand side of the identity is the unit that could bear the unconventional odd ghost parity; we let it be even.)

We introduce the 
shorthand notation
\begin{align*}
\left.\frac{\overleftarrow{\delta f}(x,[\bq],[\bq^\dagger])}{\delta q^\alpha(y)}\right|_{j^\infty_x(s)} &=
\sum\limits_{|\sigma|\geqslant0}\Bigl(-\frac{\dd}{\dd y}\Bigr)^\sigma
\left.\left(\frac{\vec{\dd}f(x,[\bq],[\bq^\dagger])}{\dd q^\alpha_\sigma}
\right)\right|_{j^\infty_x(s)}\\
\intertext{and} 
\left.\frac{\overleftarrow{\delta f}(x,[\bq],[\bq^\dagger])}{\delta q^\dagger_\beta(y)}\right|_{j^\infty_x(s)} &=
\sum\limits_{|\sigma|\geqslant0}\Bigl(-\frac{\dd}{\dd y}\Bigr)^\sigma
\left.\left(\frac{\vec{\dd}f(x,[\bq],[\bq^\dagger])}{\dd q^\dagger_{\beta,\sigma}}
\right)\right|_{j^\infty_x(s)}
\end{align*}
for the components of pure ghost parity functional derivatives. At every point of the base $M$ and for a given section $s \in \Gamma(\bpi\mathbin{{\times}_M}\Pi\wh\bpi)$ and a functional $F \in \ol{\mathfrak{M}}^n(\pi_\BV)$, we have that 
\[
\left.\frac{\overleftarrow{\delta f}(x,[\bq],[\bq^\dagger])}{\delta q^\alpha(y)}\right|_{j^\infty_x(s)} \in T^*_{s(x)}(\Pi)V_x
   \qquad\text{and}\qquad 
\left.\frac{\overleftarrow{\delta f}(x,[\bq],[\bq^\dagger])}{\delta q^\dagger_\beta(y)}\right|_{j^\infty_x(s)} \in T^*_{s(x)}(\Pi)V^\dagger_x,
\]
forcing the points~$y$ and~$x$ to coincide;
the optional presence of the parity reversion operator indicates a possibly 
odd 
ghost parity~$\gh(F)$ of the functional itself.

\begin{definition}\label{DefBVLap}
For a fixed choice of the test shift $\delta s = (\delta s^\alpha, \delta s^\d_\beta) = (\delta s^\alpha, 0) + (0, \delta s^\d_\beta)$, 
the~\emph{functional BV-\/Laplacian} is the linear operator $\lap \colon \ol{\mathfrak{N}}^n(\pi_\BV) \to \ol{\mathfrak{N}}^n(\pi_\BV)$ defined for all $F \in \ol{\mathfrak{N}}^n(\pi_\BV)$ and $s \in \Gamma(\bpi\mathbin{{\times}_M}\Pi\wh\bpi)$ by the rule
\[
(\lap F)(s) = \left\langle\diftat{\varepsilon'}0\diftat{\varepsilon''}0 F(s + \varepsilon'\overleftarrow{\delta s}^\alpha\cdot\vec{e}_\alpha + \varepsilon''\overleftarrow{\delta s}^\d_\beta\cdot\vec{e}^{\,\d,\beta})\right\rangle,
\]
where the coupling goes as follows. By taking
\[
\int_M\dif x\int_M\dif y_1\int_M\Id y_2\,
\left\langle\delta s^\alpha(y_1) 
\delta s^\d_\beta(y_2), 
\frac{\overleftarrow{\delta^2}f(x,[\bq],[\bq^\dagger])}{\delta q^\alpha(y_1)\delta q^\d_\beta(y_2)}\bigg|_{j^\infty_x(s)}\right\rangle,
\]
we obtain an integrand with the couplings $\couple{\,,}\colon T_{s(y_1)}V_{y_1}\times T^*_{s(x)}(\Pi)V_x\to\Bbbk$ 
and $\couple{\,,}\colon T_{s(y_2)}\Pi V^\d_{y_2}\times T^*_{s(x)}\Pi V^\dagger_x \to \Bbbk$. 
Let these couplings be reattached to $\couple{\,,}\colon T_{s(y_1)}V_{y_1}\times T_{s(y_2)}\Pi V^\d_{y_2}\to\Bbbk$ and $\couple{\,,}\colon T^*_{s(x)}\Pi V_x\times T^*_{s(x)}\Pi V^\dagger_x\to\Bbbk$; in fact, this surgery algorithm \emph{is} the operational definition. The coupling between dual objects~$\vec{e}_\alpha(y)$ and~$\vec{e}^{\,\d,\beta}(x)$ yields non-zero only if~$y=x$, which makes the procedure possible and we have
\begin{align*}
(\lap F)(s) &= \sum_{\alpha,\beta=1}^N
\iiint_M\dif x\,\dif y_1\,\Id y_2\ %
\delta s^\alpha(y_1)\,\delta_\alpha^\beta\,
\delta s^\d_\beta(y_2)\cdot
\frac{\overleftarrow{\delta^2}f(x,[\bq],[\bq^\dagger])}
{\delta q^\alpha(y_1)\delta q^\d_\beta(y_2)}\bigg|_{j^\infty_x(s)}
\cdot\delta(y_1-x)\cdot(y_2-x),
\intertext{which results in the conventional formula with a summation over the diagonal in $V\oplus\Pi V^\d$:}
&= \sum_{\alpha=1}^N\int_M\dif x\,\delta s^\alpha(x)\,\delta s^\d_\alpha(x)\,\frac{\overleftarrow{\delta^2}f(x,[\bq],[\bq^\dagger])}
{\delta q^\alpha \delta q^\d_\alpha }\bigg|_{j^\infty_x(s)}.
\end{align*}
In particular, if $\delta s^\alpha(x)\,\delta s^\d_\alpha(x) \equiv 1 \in \Bbbk$ (or $\equiv 1^{(\Pi)} \in \Pi\Bbbk$, which is unusual\footnote{The choice of $1^{(\Pi)} = \Pi(1)$ in the coupling preserves the ghost parity of the functional $F$ even though the second functional derivative in the definition of $\lap(F)$ does not.}), then we say that $\lap$ is the~\emph{normalized} functional BV-Laplacian.

In what follows we shall use the normalized Laplacian by default.
\end{definition}

\begin{remark}
The coupling of canonically dual components of a section $\delta s$ is performed~\emph{last} in this operational definition. Until that moment the derivatives $\dift{\varepsilon'} \longleftrightarrow \overleftarrow\delta/\delta q^\alpha(y_1)$ and $\dift{\varepsilon''} \longleftrightarrow \overleftarrow\delta/\delta q^\d_\beta(y_2)$ refer to different copies of the base $M$ and thus they can be freely swapped.
\end{remark}

The normalized functional BV-Laplacian~$\lap$ is a constituent part of a broader technique which appeals to the use of test shifts~$ \delta s \in \Gamma\bigl(T(\bpi\mathbin{{\times}_M}\Pi\wh\bpi)\bigr)$ and functional derivatives~$\diftat\varepsilon0 F(s+\varepsilon\cdot\delta s)$ for $F \in \ol{\mathfrak{N}}^n(\pi_\BV)$ instead of a use of Cartan differentials $\delta\bq$ or $\delta\bq^\d$ and of the variations $\overleftarrow{\delta_\bq}F$ or $\overleftarrow{\delta_{\bq^\d}}F$, respectively. The normalized operator~$\lap$ retains the properties of~$\lap_\BV$ on jet spaces: namely, its linearity over $\Bbbk$ on $\ol{\mathfrak{N}}^n(\pi_\BV)$, the product rule formula~\eqref{eq:laplacian-on-product}, and Proposition~\ref{thm:LaplacianDifferential}. Yet --~by a transition to the spaces of sections $\Gamma(\bpi\mathbin{{\times}_M}\Pi\wh\bpi)$ and $\Gamma\bigl(T(\bpi\mathbin{{\times}_M}\Pi\wh\bpi)\bigr)$~-- we gain a greater flexibility in our functional approach to the variations. Indeed, the (graded) permutability of functional derivatives is the key point: it resolves the obstructions which were illustrated by Counterexample~\ref{counterexample}.

From now on, we employ 
functional derivatives in the construction of the BV-Laplacian $\lap$ and the Schouten bracket $\schouten{\,,}$, for which the algorithm of reattachments and the delta function mechanisms were de-facto pronounced in section~\ref{SecSchouten}. The bracket $\schouten{\,,}$ thus also retains all its properties. On top of that, we postulate the conventional normalization $\delta s^\alpha(x)\,\delta s^\d_\alpha(x) \equiv 1 \in \Bbbk$ at each $\alpha$  and each $x \in M$, so that the resulting expressions contain neither parity-swappings nor any traces of the test-shift coefficients of the couplings $\langle \vec{e}_\alpha(x), \vec{e}^{\,\d,\beta}(y)\rangle = \delta_\alpha^\beta\cdot\delta(x-y)$ for the frames in the fibres of the BV-bundle.

\begin{example}
\label{ExCounter2}
Consider the integral functionals $F=\int q^\dagger q\, q_{x_1x_1}\,\Id x_1$ and $G=\int q^\dagger_{x_2x_2}\cos q\,\Id x_2$, c.f.\ Counterexample~\ref{counterexample} on p.~\pageref{counterexample}. Let us show that equality $\Delta\bigl(\lshad F,G\rshad\bigr)=\lshad\Delta F,G\rshad+\lshad F,\Delta G\rshad$ \emph{does hold} for these functionals whenever one uses the operational definitions of the Schouten bracket and Batalin\/--\/Vilkovisky Laplacian, hence making a proper distinction between the bases of bundles in the course of integrations by parts. (For the sake of brevity, we do not indicate the base points' congruences which occur due to the couplings. Still we explicitly refer the functionals' densities to their own domains with coordinates~$x_1$ and~$x_2$.)

We have
\begin{multline*}
\lshad F,G\rshad = \iiiint \Id x_1\Id x_2\Id y_1\Id y_2
\Bigl\langle
\Bigl(\underbrace{q^\dagger q_{xx}+\tfrac{\Id^2}{\Id y_1^2}(q^\dagger q)}_{x_1}
\Bigr) \cdot \underbrace{\langle\delta q(y_1),\delta q^\dagger(y_2)\rangle}_{+1} \cdot \tfrac{\Id^2}{\Id y_2^2}\bigl(\underbrace{\cos q}_{x_2}\bigr)
\Bigr\rangle \\
{}+ \iiiint \Id x_1\Id x_2\Id y_1\Id y_2 \Bigl\langle
\bigl(\underbrace{qq_{xx}}_{x_1}\bigr) \cdot \underbrace{\langle\delta q^\dagger(y_1),\delta q(y_2)\rangle}_{-1} \cdot \bigl(\underbrace{-q^\dagger_{xx}\,\sin q}_{x_2}\bigr)\Bigr\rangle.
\end{multline*}
Therefore, one side of the expected equality is
\begin{align*}
\Delta&\bigl(\lshad F,G\rshad\bigr)=\int\Id z_1\int\Id z_2\int\Id x_1\int\Id x_2\int\Id y_1\int\Id y_2\,\underbrace{\langle\delta q(z_1),\delta q^\dagger(z_2)\rangle}_{+1}\cdot
\underbrace{\langle\delta q(y_1),\delta q^\dagger(y_2)\rangle}_{+1}\\
{}&{}\quad
{}\cdot\Bigl\langle
\tfrac{\Id^2}{\Id z_1^2}\underbrace{(1)}_{x_1}\cdot\tfrac{\Id^2}{\Id y_2^2}\bigl(\underbrace{\cos q}_{x_2}\bigr) 
+ \underline{\underbrace{q_{xx}}_{x_1}\cdot\tfrac{\Id}{\Id y_2^2}\bigl(\underbrace{-\sin q}_{x_2}\bigr)} 
+ \tfrac{\Id^2}{\Id y_1^2}\underbrace{(1)}_{x_1}\cdot\tfrac{\Id^2}{\Id y_2^2}\bigl(\underbrace{\cos q}_{x_2}\bigr) 
+ \underline{\underline{\tfrac{\Id^2}{\Id y_1^2}\underbrace{(q)}_{x_1}\cdot\tfrac{\Id^2}{\Id y_2^2}\bigl(\underbrace{-\sin q}_{x_2}\bigr)}} 
\Bigr\rangle\\
{}&+
\int\Id z_1\int\Id z_2\int\Id x_1\int\Id x_2\int\Id y_1\int\Id y_2\,
\underbrace{\langle\delta q(z_1),\delta q^\dagger(z_2)\rangle}_{+1}\cdot{}\\
{}&{}\quad
\Bigl\langle
\underline{\underbrace{q_{xx}}_{x_1}\cdot\tfrac{\Id^2}{\Id z_2^2}\bigl(\underbrace{-\sin q}_{x_2}\bigr)}
+\underline{\underline{\tfrac{\Id^2}{\Id z_1^2}\underbrace{(q)}_{x_1}\cdot\tfrac{\Id^2}{\Id z_2^2}\bigl(\underbrace{-\sin q}_{x_2}\bigr)}}
+\bigl(\underbrace{qq_{xx}}_{x_1}\bigr)\cdot\tfrac{\Id^2}{\Id z_2^2}\bigl(\underbrace{-\cos q}_{x_2}\bigr)
\Bigr\rangle \cdot \underbrace{\langle\delta q^\dagger(y_1),\delta q(y_2)\rangle}_{-1}.
\end{align*}
The respective pairs of underlined terms cancel out and there remains only
\begin{equation}\label{EqLeft}
{}=\int{\cdots}\int\Id z_1\Id z_2\Id x_1\Id x_2\Id y_1\Id y_2
\underbrace{\langle\delta q(z_1),\delta q^\dagger(z_2)\rangle}_{+1}\cdot
\bigl\langle\bigl(\underbrace{qq_{xx}}_{x_1}\bigr)\cdot\tfrac{\Id^2}{\Id z_2^2}\bigl(\underbrace{-\cos q}_{x_2}\bigr)\bigr\rangle\cdot
\underbrace{\langle\delta q^\dagger(y_1),\delta q(y_2)\rangle}_{-1}.
\end{equation}
On the other hand, we obtain that
\[
\Delta F=\iiint\Id z_1\Id z_2\Id x_1 
\underbrace{\langle\delta q(z_1),\delta q^\dagger(z_2)\rangle}_{+1}\cdot
\bigl\langle\underbrace{1\cdot q_{xx}}_{x_1} + 1\cdot\tfrac{\Id^2}{\Id z_1^2}\underbrace{(q)}_{x_1}\bigr\rangle,
\]
which yields
\begin{multline*}
\lshad\Delta F,G\rshad=\int\Id z_1\int\Id z_2\int\Id x_1\int\Id x_2\int\Id y_1\int\Id y_2\,
\underbrace{\langle\delta q(z_1),\delta q^\dagger(z_2)\rangle}_{+1}\cdot{}\\
\Bigl\langle
\Bigl(\tfrac{\Id^2}{\Id y_1^2}\underbrace{(1)}_{x_1} +
\tfrac{\Id^2}{\Id z_1^2}\underbrace{(1)}_{x_1}\Bigr)\cdot
\tfrac{\Id^2}{\Id y_2^2}\bigl(\underbrace{\cos q}_{x_2}\bigr)
\Bigr\rangle\cdot
\underbrace{\langle\delta q(y_1),\delta q^\dagger(y_2)\rangle}_{+1}=0.
\end{multline*}
However, from the fact that the other BV-\/Laplacian,
\[
\Delta G=\iiint\Id z_1\Id z_2\Id x_2 
\underbrace{\langle\delta q(z_1),\delta q^\dagger(z_2)\rangle}_{+1}\cdot
\bigl\langle 1\cdot\tfrac{\Id^2}{\Id z_2^2}\bigl(\underbrace{-\sin q}_{x_2}\bigr)\bigr\rangle,
\]
does not contain~$q^\dagger$ so that the first half of the Schouten bracket~$\lshad F,\Delta G\rshad$ drops out, we deduce that
\begin{multline}
\lshad F,\Delta G\rshad=\int{\cdots}\int \Id z_1\Id z_2\Id x_1\Id x_2\Id y_1\Id y_2
\underbrace{\langle\delta q(z_1),\delta q^\dagger(z_2)\rangle}_{+1}\cdot{}\\
\bigl\langle\bigl(\underbrace{qq_{xx}}_{x_1}\bigr)\cdot\tfrac{\Id^2}{\Id z_2^2}\bigl(\underbrace{-\cos q}_{x_2}\bigr)\bigr\rangle\cdot
\underbrace{\langle\delta q^\dagger(y_1),\delta q(y_2)\rangle}_{-1}.
\label{EqRight}
\end{multline}
Consequently, the two sides of~\eqref{EqNotHolds}, namely, $\Delta\bigl(\lshad F,G\rshad\bigr)$ in~\eqref{EqLeft} and $\lshad\Delta F,G\rshad+\lshad F,\Delta G\rshad$ in~\eqref{EqRight}, match perfectly for the functionals~$F$ and~$G$ at~hand.
\end{example}

\begin{proposition}\label{PropBaseLapSchouten}
Let $F\in\overline{H}^{n(1+k)}\bigl(\pi_{\BV}\times T\pi_{\BV}\times\ldots\times T\pi_{\BV}\bigr)$ and $G\in\overline{H}^{n(1+\ell)}\bigl(\pi_{\BV}\times T\pi_{\BV}\times\ldots\times T\pi_{\BV}\bigr)$ be two integral functionals\textup{;} here $k,\ell\geqslant0$. Then
\begin{equation}\label{EqLapSchouten}
\lap(\schouten{F,G}) = \schouten{\lap F, G} + (-)^{\gh(F)-1}\schouten{F,\lap G}.
\end{equation}
\end{proposition}

\begin{proof}
In 
the proof 
we shall not write the evaluations at $x_1$,\ $x_2$,\ $y_1$,\ $y_2$,\ $z_1$, and~$z_2$ running over the respective copies of~$M$. Likewise, for notational convenience we also do not write the integrations 
$\int{\cdots}\int \Id z_1\Id z_2\Id x_1\Id x_2\Id y_1\Id y_2$ (see Example~\ref{ExCounter2} above). We have
\begin{align*}
\lap(\schouten{F,G})
={}&\ldifv{q^\alpha}\ldifv{q^\d_\alpha}\left(\rdifv[F]{q^\beta}\ldifv[G]{q^\d_\beta} - \rdifv[F]{q^\d_\beta}\ldifv[G]{q^\beta}\right) \\
={}& \ldifv{q^\alpha}\left(
   \frac{\overleftarrow\delta\overrightarrow\delta F}{\delta q^\d_\alpha\,\delta q^\beta}\ldifv[G]{q^\d_\beta}
   +(-)^{\gh(F)}\rdifv[F]{q^\beta}\frac{\overleftarrow\delta\overleftarrow\delta G}{\delta q^\d_\alpha\,\delta q^\d_\beta}\right. 
\left.{}-\frac{\overleftarrow\delta\overrightarrow\delta F}{\delta q^\d_\alpha\,\delta q^\d_\beta}\ldifv[G]{q^\beta}
   -(-)^{\gh(F)-1} \rdifv[F]{q^\d_\beta}\frac{\overleftarrow\delta\overleftarrow\delta G}{\delta q^\d_\alpha\,\delta q^\beta}
   \right) \displaybreak[0]\\
={}& \frac{\overleftarrow\delta\overleftarrow\delta\overrightarrow\delta F}{\delta q^\alpha\,\delta q^\d_\alpha\,\delta q^\beta}\ldifv[G]{q^\d_\beta}
   + \frac{\overleftarrow\delta\overrightarrow\delta F}{\delta q^\d_\alpha\,\delta q^\beta}\frac{\overleftarrow\delta\overleftarrow\delta G}{\delta q^\alpha\,\delta q^\d_\beta} \\
   &{} +(-)^{\gh(F)}\frac{\overleftarrow\delta\overrightarrow\delta F}{\delta q^\alpha\,\delta q^\beta}\frac{\overleftarrow\delta\overleftarrow\delta G}{\delta q^\d_\alpha\,\delta q^\d_\beta}
   +(-)^{\gh(F)}\rdifv[F]{q^\beta}\frac{\overleftarrow\delta\overleftarrow\delta\overleftarrow\delta G}{\delta q^\alpha\,\delta q^\d_\alpha\,\delta q^\d_\beta} \\
   &{} -\frac{\overleftarrow\delta\overleftarrow\delta\overrightarrow\delta F}{\delta q^\alpha\,\delta q^\d_\alpha\,\delta q^\d_\beta}\ldifv[G]{q^\beta}
   -\frac{\overleftarrow\delta\overrightarrow\delta F}{\delta q^\d_\alpha\,\delta q^\d_\beta}\frac{\overleftarrow\delta\overleftarrow\delta G}{\delta q^\alpha\,\delta q^\beta} \\
   &{} -(-)^{\gh(F)-1} \frac{\overleftarrow\delta\overrightarrow\delta F}{\delta q^\alpha\,\delta q^\d_\beta}\frac{\overleftarrow\delta\overleftarrow\delta G}{\delta q^\d_\alpha\,\delta q^\beta}
   -(-)^{\gh(F)-1} \rdifv[F]{q^\d_\beta}\frac{\overleftarrow\delta\overleftarrow\delta\overleftarrow\delta G}{\delta q^\alpha\,\delta q^\d_\alpha\,\delta q^\beta}.
\end{align*}
As a direct consequence of our approach, 
the 
functional derivatives commute -- or anticommute if they are both taken with respect to odd\/-\/parity coordinates. Applying this observation on the first and fifth terms of the calculation above, we 
immediately 
form $\schouten{\lap F, G}$. Similarly, the fourth term (in which we have to swap $\delta q^\d_\alpha$ and $\delta q^\d_\beta$ in the denominator, yielding an extra sign) and last term combine into $(-)^{\gh(F)-1}\schouten{F,\lap G}$. Thus we obtain
\begin{align}
\lap(\schouten{F,G})
={}&\frac{\overleftarrow\delta\overrightarrow\delta F}{\delta q^\d_\alpha\,\delta q^\beta}\frac{\overleftarrow\delta\overleftarrow\delta G}{\delta q^\alpha\,\delta q^\d_\beta}
+(-)^{\gh(F)}\frac{\overleftarrow\delta\overrightarrow\delta F}{\delta q^\alpha\,\delta q^\beta}\frac{\overleftarrow\delta\overleftarrow\delta G}{\delta q^\d_\alpha\,\delta q^\d_\beta} 
 - \frac{\overleftarrow\delta\overrightarrow\delta F}{\delta q^\d_\alpha\,\delta q^\d_\beta}\frac{\overleftarrow\delta\overleftarrow\delta G}{\delta q^\alpha\,\delta q^\beta} \notag\\
{}&{}
-(-)^{\gh(F)-1} \frac{\overleftarrow\delta\overrightarrow\delta F}{\delta q^\alpha\,\delta q^\d_\beta}\frac{\overleftarrow\delta\overleftarrow\delta G}{\delta q^\d_\alpha\,\delta q^\beta} 
+\schouten{\lap F, G} + (-)^{\gh(F)-1}\schouten{F,\lap G}. 
\label{eq:conventional-laplacian-schouten}
\end{align}
In the second term, the left factor is symmetric under an exchange of $\alpha$ and $\beta$, while the right factor is antisymmetric; therefore it is zero. The third term term vanishes by a similar reasoning. As to the first and fourth terms, if we set the direction of the arrows such that they all point towards the left, then the first term remains the same while the last term picks up an extra sign $(-)^{\gh(F)-1}$, after which it cancels against the first one. Thus, all four terms vanish and the claim follows.
\end{proof}

\begin{theorem}\label{ThLapSchouten}
Let $F, G \in \ol{\mathfrak{N}}^n(\pi_\BV)$ be two functionals. 
The 
Batalin\/--\/Vilkovisky Laplacian~$\lap$ satisfies
\[
   \lap(\schouten{F,G}) = \schouten{\lap F, G} + (-)^{\gh(F)-1}\schouten{F,\lap G}.
\]
\end{theorem}
\noindent In other words, the operator~
$\lap$ is a graded derivation of the variational Schouten bracket~$\schouten{\,,}$.

\begin{proof}
We prove this by induction over the number of building blocks in each argument of the Schouten bracket in the left hand side of~\eqref{EqLapSchouten}. If $F$ and $G$ both belong to $\ol{H}^*(\pi_\BV\times T\pi_\BV\times\ldots T\pi_\BV)$, then Proposition~\ref{PropBaseLapSchouten} states the assertion, which is the base of the induction. To make an inductive step, without loss of generality let us assume that the second argument of $\schouten{\,,}$ in~\eqref{EqLapSchouten} is a product of two elements from~$\ol{\mathfrak{N}}^n(\pi_\BV)$, each of them containing less multiples from~$\ol{H}^*(\pi_\BV\times T\pi_\BV\times\ldots T\pi_\BV)$ than the product. Denote such factors by $G$ and $H$ and recall that by Theorem~\ref{thm:LeibnizSchouten},
\[
\schouten{F,G\cdot H} = \schouten{F,G}\cdot H+ (-)^{(\gh(F)-1)\cdot\gh(G)}G\cdot\schouten{F,H}.
\]
Therefore, using Theorem~\ref{thm:LeibnizLaplacian} we have that
\begin{align}
\lap(\schouten{&F,G\cdot H})\nonumber\\
={}& \lap(\schouten{F,G})\cdot H + (-)^{\gh(F)+\gh(G)-1}[\![\schouten{F,G},H]\!] + (-)^{\gh(F)+\gh(G)-1}\schouten{F,G}\cdot\lap H\nonumber\\
& + (-)^{(\gh(F)-1)\gh(G)}\left(\lap G\cdot\schouten{F,H} + (-)^{\gh(G)}[\![ G,\schouten{F,H}]\!] + (-)^{\gh(G)}G\cdot\lap(\schouten{F,H})\right).\nonumber
\intertext{Using the inductive hypothesis in the first and last terms of the right hand side in the above formula, we continue the equality and obtain}
={}&\lap(\schouten{F,G})\cdot H + (-)^{\gh(F)-1}\schouten{F,\lap G}\cdot H + (-)^{\gh(F)+\gh(G)-1}[\![\schouten{F,G},H]\!] \nonumber\\
&+ (-)^{\gh(F)\gh(G)}[\![ G,\schouten{F,H}]\!] + (-)^{\gh(F)+\gh(G)-1}\schouten{F,G}\cdot\lap H \nonumber\\
&+ (-)^{(\gh(F)-1)\gh(G)}\lap G\cdot\schouten{F,H} + (-)^{\gh(F)\gh(G)}G\cdot\schouten{\lap F,H} \nonumber\\
& + (-)^{\gh(F)\gh(G)+\gh(F)-1}G\cdot\schouten{F,\lap H}.\label{EqIndStep}
\end{align}
On the other hand, let us expand the formula
\[
\schouten{\lap F, G\cdot H} + (-)^{\gh(F)-1}\schouten{F,\lap(G\cdot H)},
\]
which is the right hand side of~\eqref{EqLapSchouten} in the inductive claim. We obtain
\begin{align}
={}& \schouten{\lap F, G}\cdot H + (-)^{(\gh(\lap F)-1)\gh(G)}G\cdot\schouten{\lap F,H} \nonumber\\
&+ (-)^{\gh(F)-1}[\![ F,\lap G\cdot H  + (-)^{\gh(G)}\schouten{G,H} + (-)^{\gh(G)}G\cdot\lap H]\!] \nonumber\\
={}& \schouten{\lap F, G}\cdot H + (-)^{\gh(F)\gh(G)}G\cdot\schouten{\lap F,H} + (-)^{\gh(F)-1}\schouten{F,\lap G}\cdot H \nonumber\\
&+ (-)^{\gh(F)-1}(-)^{(\gh(F)-1)(\gh(G)-1)}\lap G\cdot\schouten{F,H} + (-)^{\gh(F)-1}(-)^{\gh(G)}[\![ F,\schouten{G,H}]\!] \nonumber\\
&+ (-)^{(\gh(F)-1)\gh(G)}\schouten{F,G}\cdot\lap H \nonumber\\
&+ (-)^{\gh(F)-1}(-)^{\gh(G)}(-)^{(\gh(F)-1)\gh(G)}G\cdot\schouten{F,\lap H}. \label{EqClaimExpand}
\end{align}
Comparing~\eqref{EqClaimExpand} with~\eqref{EqIndStep}, which was derived from the inductive hypothesis, we see that all terms match except for
\[
(-)^{\gh(F)+\gh(G)-1}[\![\schouten{F,G},H]\!] + (-)^{\gh(F)\gh(G)}[\![ G,\schouten{F,H}]\!]
\]
from~\eqref{EqIndStep} versus
\[
(-)^{\gh(F)+\gh(G)-1}[\![ F,\schouten{G,H}]\!]
\]
from~\eqref{EqClaimExpand}. However, these three terms constitute the Jacobi identity~\eqref{EqJacobiSchouten} for the Schouten bracket. Namely, we have that (c.f.~\cite{Lorentz12})
\[
[\![ F,\schouten{G,H}]\!] = [\![\schouten{F,G},H]\!] + (-)^{(\gh(F)-1)(\gh(G)-1)}[\![ G,\schouten{F,H}]\!],
\]
so that by multiplying both sides of the identity by $(-)^{\gh(F)+\gh(G)-1}$, we fully balance~\eqref{EqIndStep} and~\eqref{EqClaimExpand}. This completes the inductive step and concludes the proof.
\end{proof}

\begin{lemma}\label{LBVOperatorDifferential}
The linear operator
\[
\Delta\colon\overline{H}^{n(1+k)}\bigl(\pi_\BV\times T\pi_\BV\times\ldots\times T\pi_\BV\bigr)\longrightarrow\overline{H}^{n(2+k)}\bigl(\pi_\BV\times T\pi_\BV\times\ldots\times T\pi_\BV\bigr)
\]
is a differential for every $k\geqslant0$.
\end{lemma}

\begin{proof}
Let $H\in\overline{H}^*\bigl(\pi_\BV\times T\pi_\BV\times\ldots\times T\pi_\BV\bigr)$ be an integral functional.
Applying the BV-\/Laplacian~$\Delta$ twice and again, not writing the integration signs, 
we obtain
\[
\lap^2(H) = \difv{q^\alpha}\difv{q^\d_\alpha}\difv{q^\beta}\difv{q^\d_\beta}(H)
 = \difv{q^\alpha}\difv{q^\beta}\difv{q^\d_\alpha}\difv{q^\d_\beta}(H),
\]
where we have swapped the middle two functional derivatives in the right hand side. This is the composition of an expression which is symmetric in~$\alpha$ and~$\beta$ (the left two functional derivatives) and an expression which is antisymmetric in~$\alpha$ and~$\beta$ (the right two functional derivatives). Therefore it is zero.
\end{proof}

\begin{theorem}\label{ThBVDifferential}
The Batalin\/--\/Vilkovisky Laplacian~$\lap$ is a differential\textup{:} for all $H \in \ol{\mathfrak{N}}^n(\pi_\BV)$ we have 
\[
\lap^2(H) = 0.
\]
\end{theorem}

\begin{proof}
We prove Theorem~\ref{ThBVDifferential} by induction of the number of building blocks from~$\ol{H}^*\bigl(\pi_\BV\times T\pi_\BV\times\ldots\times T\pi_\BV\bigr)$ in the argument $H \in \ol{\mathfrak{N}}^n(\pi_\BV)$ of~$\lap^2$. If $H \in \ol{H}^*\bigl(\pi_\BV\times T\pi_\BV\times\ldots\times T\pi_\BV\bigr)$ itself is an integral functional, then by Lemma~\ref{LBVOperatorDifferential}
there remains nothing to prove. Suppose now that $H = F\cdot G$ for some $F,G \in \ol{\mathfrak{N}}^n(\pi_\BV)$. Then Theorem~\ref{thm:LeibnizLaplacian} yields that
\begin{align*}
\lap^2(F\cdot G) ={}& \lap\left(\lap F\cdot G + (-)^{\gh(F)}\schouten{F,G} + 
 (-)^{\gh(F)}F\cdot\lap G\right).
\intertext{Using Theorem~\ref{thm:LeibnizLaplacian} again and also Theorem~\ref{ThLapSchouten}, we continue the equality:}
={}& \lap^2F\cdot G + (-)^{\gh(\lap F)}\schouten{\lap F,G} + (-)^{\gh(\lap F)}\lap F\cdot\lap G + (-)^{\gh(F)}\schouten{\lap F,G} \\
& + (-)^{\gh(F)}(-)^{\gh(F)-1}\schouten{F,\lap G} + (-)^{\gh(F)}\lap F\cdot\lap G \\
&+ (-)^{\gh(F)}(-)^{\gh(F)}\schouten{F,\lap G} + (-)^{\gh(F)}(-)^{\gh(F)}F\cdot\lap^2G.
\end{align*}
By the inductive hypothesis, the first and last terms in the above formula vanish; taking into account that $\gh(\lap F) = \gh(F)-1$ in $\BBZ_2$, the terms with $\lap F\cdot\lap G$ cancel against each other, as do the terms containing $\schouten{\lap F,G}$ and $\schouten{F,\lap G}$. The proof is complete.
\end{proof}

\section{The quantum master equation}\label{SecMaster}
In this last section we inspect the conditions upon functionals $F \in \ol{\mathfrak{N}}^n(\pi_\BV)$ under which the Feynman path integrals $\int_{\Gamma(\bpi)}[Ds]\,F([s],[s^\d])$ are (infinitesimally) independent of the unphysical, odd-parity antifields $s^\d \in \Gamma(\Pi\wh\bpi)$. The derivation of such a condition (see equation~\eqref{EqLaplace} below) relies on an extra assumption of the translation invariance of a measure in the path integral. It must be noted, however, that we do not define Feynman's integral here and do not introduce that measure which essentially depends on the agreement about the classes of `admissible' sections $\Gamma(\bpi)$ or $\Gamma(\bpi\mathbin{{\times}_M}\Pi\wh\bpi)$. Consequently, our reasoning is to some extent heuristic.

The basics of path integration, which we recall here for consistency, are standard: they illustrate how the geometry of the BV-Laplacian works in practice. We draw the experts' attention only to the fact that in our notation $\Psi$ is not the gauge fixing fermion $\boldsymbol\Psi$ such that $s^\d(x) = \Pi(\delta\boldsymbol\Psi/\delta q(x))$ but it yields the infinitesimal shift $\dot{q}^\d = \delta\Psi/\delta q$ of the anti-objects; we also note that the preservation of parity is not mandatory here and thus an even-parity $\Psi \in \ol{H}^n(\bpi) \hookrightarrow \ol{H}^n(\pi_\BV)$ is a legitimate choice.

Let $F \in \ol{\mathfrak{N}}^n(\pi_\BV)$ be a functional and $\Psi \in \ol{H}^n(\bpi) \hookrightarrow \ol{H}^n(\pi_\BV)$ be an integral functional which, by assumption, is constant along the antifields: $\Psi(s^\alpha,s^\d_\beta) = \Psi(s^\alpha,t^\d_\beta)$ for any sections $\{s^\alpha\} \in \Gamma(\bpi)$ and $\{s^\d_\beta\},\{t^\d_\beta\} \in \Gamma(\Pi\wh\bpi)$. We investigate under which conditions the path integral $\int_{\Gamma(\bpi)}[Ds^\alpha]\,F(s^\alpha,s^\d_\beta) \colon \Gamma(\Pi\wh\bpi) \to \Bbbk$ is infinitesimally independent of a choice of the antifields:
\begin{align}\label{eq:PathIntegralIndependent}
\diftat\varepsilon0\int_{\Gamma(\bpi)}[Ds^\alpha]\,F\left(s^\alpha,s^\d_\beta + \varepsilon\,\difv[\Psi]{q^\beta(x)}\bigg|_{(s^\alpha,s^\d_\beta)}\right) = 0 \quad\text{for all $s^\d \in \Gamma(\Pi\wh\bpi)$.}
\end{align}
Note that this formula makes sense because the bundles $\bpi$ and $\wh\bpi$ are dual so that a variational covector in the geometry of $\bpi$ acts as a shift vector in the geometry of $\wh\bpi$; we let $\varepsilon$ be an odd-parity parameter. The left hand side of~\eqref{eq:PathIntegralIndependent} equals
(here and in what follows we proceed over the building blocks of~$F$ by the graded Leibniz rule)
\[
\int_{\Gamma(\bpi)}[Ds^\alpha]\int_M\dif x\,\rdifv[\Psi]{q^\alpha(x)}\bigg|_{(s^\alpha,s^\d_\beta)}\cdot\ldifv[F]{q^\d_\alpha(x)}\bigg|_{(s^\alpha,s^\d_\beta)}, \qquad s^\d\in\Gamma(\Pi\wh\bpi).
\]
Take any auxiliary section $\delta s = (\delta s^\alpha, \delta s^\d_\beta) \in \Gamma\bigl(T(\bpi\mathbin{{\times}_M}\Pi\wh\bpi)\bigr)$ normalized by $\delta s^\alpha(x)\cdot\delta s^\d_\alpha(x) \equiv 1$ at every $x\in M$ for each $\alpha = 1,\dots,m + m_1+\cdots+m_\lambda = N$ and blow up the scalar integrand to a pointwise contraction of dual object taking their values in the fibres $T_{s(x)}V_x$ and $T_{s^\dagger(x)}\Pi V^\d_x$ of $T(\bpi\mathbin{{\times}_M}\Pi\wh\bpi)$ over $x \in M$: for $s = (s^\alpha, s^\d_\beta)$ we have
\begin{align*}
\int_M\dif x&\int_M\dif y\,\rdifv[\Psi]{q^\alpha(x)}\bigg|_s
\cdot\delta^\alpha_\beta\,\delta(x-y)\cdot
\ldifv[F]{q^\d_\beta(y)}\bigg|_s \nonumber\\
&= \int_M\dif x\int_M\dif y\,\rdifv[\Psi]{q^\alpha(x)}\bigg|_s\cdot\delta s^\alpha(x)\,\langle\vec{e}_\alpha(x),\vec{e}^{\,\d,\beta}(y)\rangle\,\delta s^\d_\beta(y)\cdot\ldifv[F]{q^\d_\beta(y)}\bigg|_s.
\end{align*}
In fact, the integrand refers to a definition of the evolutionary vector field $\bs{Q}^\Psi$ such that $\bs{Q}^\Psi(F) \cong \schouten{\Psi, F}$ modulo integration by parts in the building blocks of $F$, c.f.~\cite{Lorentz12}. Due to a special choice of the dependence of $\Psi$ on $q$ only, this is indeed the Schouten bracket $\schouten{\Psi, F}$.

To rephrase the indifference of the path integral to a choice of $\Psi$ in terms of an equation upon the functional $F$ alone, we perform integration by parts in Feynman's integral. For this we employ the translation invariance $[Ds] = [D(s-\mu\cdot\delta s)]$ of the functional measure.

\begin{lemma}\label{LTI}
Let $H \in \ol{\mathfrak{N}}^n(\pi_\BV)$ be a functional and $\delta s \in \Gamma(T\bpi) \hookrightarrow \Gamma\bigl(T(\bpi\mathbin{{\times}_M}\Pi\wh\bpi)\bigr)$ be a shift section. Then we have that
\[
\int_{\Gamma(\bpi)}[Ds^\alpha]\int_M\dif x\,\delta s^\alpha(x)\,\ldifv[H]{q^\alpha(x)}\bigg|_{(s^\alpha,s^\d_\beta)} = 0,
\]
where the section $s^\d \in \Gamma(\Pi\wh\bpi)$ is a parameter.
\end{lemma}
\begin{proof}
Indeed,
\begin{align*}
0 &= \diftat\mu0\int_{\Gamma(\bpi)}[Ds^\alpha]\,H(s^\alpha,s^\d_\beta),
\intertext{because the integral contains no parameter $\mu \in \Bbbk$. We continue the equality:}
&= \diftat\mu0\int_{\Gamma(\bpi)}[D(s^\alpha-\mu\,\delta s^\alpha)]\,H(s^\alpha,s^\d_\beta) \\
&= \diftat\mu0\int_{\Gamma(\bpi)}[Ds^\alpha]\,H(s^\alpha + \mu\,\delta s^\alpha,s^\d_\beta)
 = \int_{\Gamma(\bpi)}[Ds^\alpha]\diftat\mu0H(s^\alpha + \mu\,\delta s^\alpha,s^\d_\beta),
\end{align*}
which yields the helpful formula in the Lemma's assertion.
\end{proof}

Returning to the functionals $\Psi$ and $F$ and denoting $G(s) := \diftat\ell0 F(s + \ell\cdot\overleftarrow{\delta s}^\d_\beta\cdot\vec{e}^{\,\d,\beta})$, we use the Leibniz rule for the derivative of $H = \Psi\cdot G$:
\begin{align*}
\diftat\mu0(\Psi\cdot G)(s + \mu\cdot\overleftarrow{\delta s}^\alpha\cdot\vec{e}_\alpha)
 ={}& \int\dif x\,\delta s^\alpha(x)\,\ldifv[\Psi]{q^\alpha(x)}\bigg|_s\cdot G(s) \\
    & + \Psi(s)\cdot\int\dif x\,\delta s^\alpha(x)\,\ldifv[G]{q^\alpha(x)}\bigg|_s.
\end{align*}
Because the path integral over $[Ds^\alpha]$ of the entire expression vanishes by Lemma~\ref{LTI}, we infer that the path integrals of the two terms are opposite. The integral of the first term equals the initial expression for the path integral over $F$, i.e., the left\/-\/hand side of equation~\eqref{eq:PathIntegralIndependent}. Consequently, if
\begin{align}\label{ObserveInPath}
\int_{\Gamma(\bpi)}[Ds^\alpha]\,\Psi(s^\alpha)\cdot\lap F(s^\alpha,s^\d_\beta) = 0
\end{align}
for $\{s^\d_\beta\} \in \Gamma(\Pi\wh\bpi)$ and for all $\Psi \in \ol{H}^n(\bpi) \hookrightarrow \ol{H}^n(\pi_\BV)$, then the path integral over $F$ is infinitesimally independent of the antifields $\{s^\d_\beta\} \in \Gamma(\Pi\wh\bpi)$.

The condition
\begin{equation}\label{EqLaplace}
\lap F = 0
\end{equation}
is sufficient for equation~\eqref{ObserveInPath}, and therefore equation~\eqref{eq:PathIntegralIndependent}, to hold. By specifying a class $\Gamma(\bpi\mathbin{{\times}_M}\Pi\wh\bpi)$ of admissible sections of the BV-bundle for a concrete field model, and endowing that space of sections with a suitable metric, one could reinstate a path integral analogue of the main lemma in the calculus of variations and then argue that the condition $\lap F = 0$ is also necessary.

Summarizing, whenever equation~\eqref{EqLaplace} holds, one can assign arbitrary admissible values to the odd-parity antifields; for example, one can let $s^\d_\beta(x) = \Pi(\delta\bs\Psi/\delta q^\alpha(x))$ for a gauge-fixing integral $\bs\Psi \in \ol{H}^n(\bpi)$. This choice is reminiscent of the substitution principle, see~\cite{Lorentz12} and~\cite{Olver}.

\begin{proposition}
The $\Bbbk$-linear space of solutions~$F \in \ol{\mathfrak{N}}^n(\pi_\BV)$ to the equation~\eqref{EqLaplace} is a graded Poisson algebra with respect to the Schouten bracket $\schouten{\,,}$ and the multiplication~$\cdot$ in $\ol{\mathfrak{N}}^n(\pi_\BV)$.
\end{proposition}

\begin{proof}
The property $\lap F = 0$ and $\lap G = 0$ implies $\lap\schouten{F,G} = 0$ by Theorem~\ref{ThLapSchouten}.
\end{proof}

The Laplace equation~\eqref{EqLaplace} ensures the infinitesimal independence from non-physical anti-objects for path integrals of functionals over physical fields -- not only in the classical BV-geometry of the bundle $\bpi\mathbin{{\times}_M}\Pi\wh\bpi$, but also in the quantum setup, whenever all objects are tensored with formal power series $\Bbbk[[\hbar,\hbar^{-1}]]$ in the Planck constant $\hbar$. It is accepted that each quantum field $s^\hbar$ contributes to the expectation value of a functional~$\mathcal{O}^\hbar$ with the factor~$\exp({i}S^\hbar_\BV(s^\hbar)/{\hbar})$, where~$S^\hbar_\BV$ is the quantum BV-action of the model. Solutions $\mathcal{O}^\hbar$ of the equation $\lap(\mathcal{O}^\hbar\cdot\exp({i}S^\hbar_\BV/{\hbar})) = 0$ are the~\emph{observables}. In particular, the postulate that the unit $1 \colon s \mapsto 1 \in \Bbbk$ is averaged to unit by the Feynman integral of $1\cdot\exp({i}S^\hbar_\BV(s^\hbar)/{\hbar})$ over the space of quantum fields $s^\hbar$ normalizes the integration measure and constrains the quantum BV-action by the quantum master equation (see also, e.g.,~\cite{BV1,BV2,HT}). 

\begin{proposition}\label{thm:QME}
   Let $S_\BV^\hbar$ be the even \textup{(}i.e., it has a density that has an even number of anticommuting coordinates in it\textup{)} quantum BV-action. If the identity $\lap(\exp(i S_\BV^\hbar/\hbar)) = 0$ holds, then $S_\BV^\hbar$ satisfies the quantum master equation\textup{:}
      \begin{align}\label{eq:QME}
         \frac12\schouten{S_\BV^\hbar,S_\BV^\hbar} = i\hbar\lap S_\BV^\hbar.
      \end{align}
\end{proposition}
\noindent We will need the following two lemmas.
\begin{lemma}\label{thm:SchoutenPower}
   Let $F \in \ol{H}^n(\pi_\BV)$ be an even integral functional, let $G \in \ol{\mathfrak{M}}^n(\pi_\BV)$ be another functional, and let $n \in \mathbb{N}_{\geq1}$. Then
   \[
      \schouten{G,F^n} = n\schouten{G,F}F^{n-1}.
   \]
\end{lemma}
\begin{proof}
   We use induction on Theorem~\ref{thm:LeibnizSchouten}. Note that all signs vanish since $F$ is even, meaning that whenever $F$ is multiplied with any other integral functional, the factors may be freely swapped without this resulting in minus signs. For $n=1$ the statement is trivial. Suppose the formula holds for some $n \in \mathbb{N}_{>1}$, then
   \[
      \schouten{G,F^{n+1}}
       = \schouten{G,F\cdot F^n}
       = \schouten{G,F}F^n + F\schouten{G,F^n}
       = \schouten{G,F}F^n + nF\schouten{G,F}F^{n-1}
       = (n+1)\schouten{G,F}F^n,
   \]
   so that the statement also holds for $n+1$.
\end{proof}

\begin{lemma}
   Let $F \in \ol{H}^n(\pi_\BV)$ be an even integral functional, and let $n \in \mathbb{N}_{\geq2}$. Then
   \[
      \lap(F^n) = n(\lap F)F^{n-1} + \frac12n(n-1)\schouten{F,F}F^{n-2}.
   \]
\end{lemma}
\begin{proof}
   We use induction and the previous lemma. For $n=2$ the formula clearly holds by Theorem~\ref{thm:LeibnizLaplacian}. Suppose that it holds for some $n \in \mathbb{N}_{>2}$, then
   \begin{align*}
      \lap(F^{n+1}) &= \lap(F\cdot F^n) 
       = (\lap F)F^n + \schouten{F,F^n} + F\lap(F^n) \\
      &= (\lap F)F^n + n\schouten{F,F}F^{n-1} + nF(\lap F)F^{n-1} + \frac12n(n-1)F\schouten{F,F}F^{n-2} \\
      &= (n+1)(\lap F)F^n + \frac12(n+1)n\schouten{F,F}F^{n-1},
   \end{align*}
   so that the statement also holds for $n+1$.
\end{proof}

\begin{proof}[Proof of Proposition~\ref{thm:QME}]
   For convenience, we set $F = \frac{i}{\hbar}S_\BV^\hbar$. Then
   \begin{align*}
      0 &= \lap(\exp F) 
         = \lap\left(\sum_{n=0}^\infty \frac1{n!}F^n\right)
         = \sum_{n=0}^\infty \frac1{n!}\lap(F^n) \\
        &= \sum_{n=0}^\infty \frac{n}{n!}(\lap F)F^{n-1} + \sum_{n=0}^\infty\frac{1}{2n!}n(n-1)\schouten{F,F}F^{n-2} \\
        &= (\lap F)\sum_{n=1}^\infty \frac{1}{(n-1)!}F^{n-1} + \frac12\schouten{F,F}\sum_{n=2}^\infty\frac{1}{(n-2)!}F^{n-2} \\
        &= \left(\lap F + \frac12\schouten{F,F}\right)\exp F 
         = \left(\frac{i}{\hbar}\lap S_\BV^\hbar - \frac1{2\hbar^2}\schouten{S_\BV^\hbar,S_\BV^\hbar}\right)\exp\left(\frac{i}{\hbar}S_\BV^\hbar\right),
   \end{align*}
   from which the result follows.
\end{proof}

\begin{proposition}
   If the even functionals $\mathcal{O}$ and $S_\BV^\hbar$ are such that $\lap(\mathcal{O}\exp(i S_\BV^\hbar/\hbar)) = 0$ and $\lap(\exp(i S_\BV^\hbar/\hbar)) = 0$ hold, respectively, then $\mathcal{O}$ satisfies
   \begin{align}\label{eq:def-Omega}
      \Omega(\mathcal{O}) := \schouten{S_\BV^\hbar, \mathcal{O}} - i\hbar \lap \mathcal{O} = 0.
   \end{align}
\end{proposition}
\begin{proof}
   For convenience, let us set $F = \frac{i}{\hbar}S_\BV^\hbar$ again. We first calculate, using Lemma~\ref{thm:SchoutenPower},
   \[
      \schouten{\mathcal{O},\exp F}
      = \sum_{n=0}^\infty\frac1{n!}\schouten{\mathcal{O},F^n}
      = \sum_{n=0}^\infty\frac{n}{n!}\schouten{\mathcal{O},F}F^{n-1}
      = \schouten{\mathcal{O},F}\exp F.
   \]
   Then
   \begin{align*}
      0 &= \lap(\mathcal{O}\exp F) 
         = (\lap \mathcal{O})\exp F + \schouten{\mathcal{O}, \exp F} + \mathcal{O}\lap(\exp F) \\
        &= \big((\lap \mathcal{O}) + \schouten{\mathcal{O}, F}\big)\exp F 
         = \left((\lap \mathcal{O}) + \frac{i}{\hbar}\schouten{\mathcal{O}, S_\BV^\hbar}\right)\exp\left(\frac{i}{\hbar}S_\BV^\hbar\right),
   \end{align*}
   from which the assertion follows.
\end{proof}

\begin{proposition}
Let $\mathcal{O} \in \ol{\mathfrak{M}}^n(\pi_\BV)$ be an integral functional, and let the even functional $S_\BV^\hbar\in \ol{H}^n(\pi_\BV)$ satisfy the quantum master equation~\eqref{eq:QME} for the BV-Laplacian $\lap$. Then the operator $\Omega$, defined in~\eqref{eq:def-Omega}, squares to zero\textup{:} 
\[ \Omega^2(\mathcal{O}) = 0. \]
\end{proposition}
\begin{proof}
We calculate, using Theorem~\ref{ThLapSchouten},
\begin{align*}
\Omega^2(\mathcal{O}) 
&=[\![S_\BV^\hbar,\schouten{S_\BV^\hbar,\mathcal{O}}-i\hbar\lap \mathcal{O}]\!] - i\hbar\lap\big(\schouten{S_\BV^\hbar,\mathcal{O}} - i\hbar\lap \mathcal{O}\big) \\
&=[\![S_\BV^\hbar,\schouten{S_\BV^\hbar,\mathcal{O}}]\!] - i\hbar\schouten{S_\BV^\hbar,\lap \mathcal{O}} - i\hbar\schouten{\lap S_\BV^\hbar,\mathcal{O}} + i\hbar\schouten{S_\BV^\hbar,\lap \mathcal{O}} + (i\hbar)^2\lap^2\mathcal{O}.
\end{align*}
The last term vanishes identically by Theorem~\ref{ThBVDifferential}, while the second term cancels against the fourth term. Using the Jacobi identity~\eqref{EqJacobiSchouten} for the Schouten bracket on the first term, we obtain:
\begin{align*}
\Omega^2(\mathcal{O}) 
&= {\textstyle\frac12}[\![\schouten{S_\BV^\hbar,S_\BV^\hbar},\mathcal{O}]\!] - i\hbar\schouten{\lap S, \mathcal{O}} 
 = [\![{\textstyle\frac12}\schouten{S_\BV^\hbar,S_\BV^\hbar}-i\hbar\lap S_\BV^\hbar, \mathcal{O}]\!] 
 = 0,
\end{align*}
by the quantum master equation for $S_\BV^\hbar$.
\end{proof}

The introduction of the quantum BV-differential $\Omega$ concludes our overview of the geometric properties of the Batalin-Vilkovisky Laplacian; yet the story continues with the quantum BV-cohomology theory.

\subsection*{Conclusion}
We have outlined a natural class of geometries and described the spaces of admissible functionals (almost none of which are integral functionals except the layer of building blocks such as the action~$S$ of the model or the corresponding BV-\/action~$S_\BV$) such that the regularization of the BV-\/Laplacian~$\lap_\BV$ is automatic. The operational definitions of~$\schouten{\,,}$ and~$\lap_\BV$ and the validity of their intrinsic properties 
justify the traditional algorithms which are employed in a vast part of the literature to cope with the regularization of~$\delta$-\/functions or infinite constants in this context. At the same time, we expect that the self\/-\/regularization mechanism, which we have described and used in this note, can be adapted to and work in a wider set of geometries and formalism that involve double variations or variational derivatives of functionals with respect to pairs of canonically conjugate variables.

We finally remark that the jet\/-\/bundle setup and differential calculus do not appeal to commutativity in the models (e.g., to the $\mathbb{Z}_2$-\/graded commutativity if~$\pi = \bigl(\pi^{\overline{0}}|\pi^{\overline{1}}\bigr)$ is the superbundle of physical fields). In fact, the entire approach remains valid in the formal noncommutative geometry based on the calculus of cyclic\/-\/invariant words (see~\cite{Lorentz12,ArthemyChristmas} and references therein).

\subsubsection*{Acknowledgments}
The authors thank G.~Felder for stimulating remarks.
The first author thanks the organizing committee of the international workshop SQS'11 `Supersymmetry and Quantum Symmetries' (July 18--23, 2011; JINR~Dubna, Russia) for helpful discussions. The research of the first author was partially supported by NWO~VENI grant~639.031.623 (Utrecht) and JBI~RUG project~103511 (Groningen).

\bibliographystyle{sigma}
\bibliography{refs}

\end{document}